\documentclass[12pt]{article}
\usepackage{amsthm}
\usepackage{amsmath,amsfonts,amssymb,rotating,colordvi}
\usepackage{color}
\usepackage[unicode=true,
 bookmarks=true,bookmarksnumbered=true,bookmarksopen=true,bookmarksopenlevel=2,
 breaklinks=false,pdfborder={0 0 1},backref=false,colorlinks=true]
 {hyperref}

\usepackage{tikz}
\usetikzlibrary{arrows,shapes,positioning}
\usetikzlibrary{decorations.markings}
\tikzstyle arrowstyle=[scale=1]
\tikzstyle directed=[postaction={decorate,decoration={markings,mark=at position 0.6 with {\arrow[arrowstyle]{stealth};}}}]
\tikzstyle reverse directed=[postaction={decorate,decoration={markings,mark=at position 0.4 with {\arrowreversed[arrowstyle]{stealth};}}}]
 
\hypersetup{pdftitle={Word-representability of subdivisions of triangular grid graphs},
 pdfauthor={Zongqing Chen, Sergey Kitaev and Brian Y. Sun,...},
 pdfsubject={ },
 pdfkeywords={ },
 linkcolor=black,citecolor=black,urlcolor=blue,filecolor=blue,pdfpagelayout=OneColumn,pdfnewwindow=true,pdfstartview=XYZ,plainpages=false}

 \newtheorem{thm}{Theorem}[section]

\newtheorem{lem}[thm]{Lemma}
\newtheorem{prop}[thm]{Proposition}

\newtheorem{remark}[thm]{Remark}
\numberwithin{equation}{section}

\newdimen\Squaresize \Squaresize=11pt
\newdimen\Thickness \Thickness=0.7pt
\def\Square#1{\hbox{\vrule width \Thickness
   \vbox to \Squaresize{\hrule height \Thickness\vss
    \hbox to \Squaresize{\hss#1\hss}
   \vss\hrule height\Thickness}
\unskip\vrule width \Thickness} \kern-\Thickness}

\def\Vsquare#1{\vbox{\Square{$#1$}}\kern-\Thickness}

\def\moins{\raise 1pt\hbox{{$\scriptstyle -$}}}

\begin{document}

\begin{center}
\textbf{\large{Word-representability of subdivisions of triangular grid graphs}}\textbf{ }
\par\end{center}

\begin{center}
Zongqing Chen$^{a}$, Sergey Kitaev$^{b}$, Brian Y. Sun$^{c}$\\[6pt]
\par\end{center}

\begin{center}
$^{a,c}$Center for Combinatorics, LPMC-TJKLC\\
 Nankai University, Tianjin 300071, P. R. China\\
$^{b}$Department of Computer and Information Sciences,\\ 
 University of Strathclyde, Glasgow, G1 1XH, UK
\par\end{center}

\begin{center}
Email: $^{a}$\texttt{zqchern@163.com}, $^{b}$\texttt{sergey.kitaev@cis.strath.ac.uk},
$^{c}$\texttt{brian@mail.nankai.edu.cn} 
\par
\end{center}

\par
\noindent \textbf{Abstract.} 
A graph $G=(V,E)$ is {\em word-representable} if there exists a word $w$ 
over the alphabet $V$ such that letters $x$ and $y$ alternate in $w$ if and only if $(x,y)\in E$. A {\it triangular grid graph} is a subgraph of a tiling of the plane with equilateral triangles defined by a finite number of triangles, called {\em cells}. A {\it subdivision} of a triangular grid graph is replacing some of its cells by plane copies of the complete graph $K_4$. 

Inspired by a recent elegant result of Akrobotu et al., 
who classified word-representable triangulations of grid graphs related to convex polyominoes, we characterize word-representable subdivisions of triangular grid graphs. A key role in the characterization is played by {\em smart orientations} introduced by us in this paper. As a corollary to our main result, we obtain that any subdivision of  boundary triangles in the {\em Sierpi\'{n}ski gasket graph} is word-representable. \\


\section{Introduction}

Let $G=(V,E)$ be a simple (i.e. without loops and multiple edges) undirected graph with 
the vertex set $V$ and the edge set $E$. We say that $G$ is {\em word-representable} if there exists a word $w$ 
over the alphabet $V$ such that letters $x$ and $y$ alternate in $w$ if and only if $(x,y)\in E$ for
any $x\neq y$.

The notion of word-representable graphs has its roots in algebra, 
where a prototype of these graphs was used by Kitaev and Seif to study
the growth of the free spectrum of the well-known
{\em Perkins semigroup} \cite{Kitaev08b}. 

Recently, a number of (fundamental) results on word-representable graphs 
were obtained in the literature; for example, see \cite{Akrobotu}, \cite{Collins14}, 
\cite{Halldorsson11}, \cite{Halldorsson15}, \cite{Kitaev08a}, \cite{Kitaev11c}, and \cite{Kitaev11}. In 
particular, Halld\'{o}rsson et al.~\cite{Halldorsson15} have shown that a graph is word-representable 
if and only if it admits a {\em semi-transitive orientation} (to be defined in Section~\ref{sec2}), which, among other important corollaries, implies that all 3-colorable graphs are word-representable. The theory of word-representable graphs is the main subject of the upcoming book~\cite{KitLoz}.

The {\em triangular tiling graph} $T^{\infty}$ (i.e., the two-dimensional {\em triangular grid}) is the {\em Archimedean tiling} $3^6$ is more common introduced in \cite{Ren87} and \cite{Gordon08}. By a {\it triangular grid graph} $G$ in this paper we mean a graph obtained from $T^{\infty}$ as follows. Specify a number of triangles, called {\em cells}, in $T^{\infty}$. The edges of $G$ are then all the edges surrounding the specified cells, while the vertices of $G$ are the endpoints of the edges (defined by intersecting lines in $T^{\infty}$). We say that the specified cells, along with any other cell whose all edges are from $G$, {\em belong} to $G$. 
Any triangular grid graph is 3-colorable, and thus it is word-representable \cite{Halldorsson15}. We consider non-3-colorable graphs obtained from triangular grid graphs by applying the operation of subdivision which is defined in the sequel.

{\it Subdividing} a cell of a triangular grid graph means subdividing it into three parts by placing a vertex in the center of the cell and making it adjacent to the three cell's vertices. A subdivision of a triangular grid graph is obtained by subdividing a number of specified cells in $G$.

Recently, Akrobotu at el.~\cite{Akrobotu} proved that a triangulation of the graph $G$ associated with a convex polyomino is word-representable if and only if $G$ is 3-colorable. Inspired by this elegant result, in the paper in hands, we characterized word-representable subdivisions of triangular grid graphs.

The paper is organized as follows. In Section~\ref{sec2} some necessary definitions, notation and known results are given. In Section~\ref{sec3} we discuss the minimal non-word-representable subdivision of a triangular grid graph, i.e. the graph, which is
an induced subgraph of any non-word-representable subdivision. In Section~\ref{sec4} we state and prove our main result (Theorem~\ref{subdivtrigrid}) saying that a subdivision of a triangular grid graph is word-representable if and only if it has no {\em interior cell} subdivided. Theorem~\ref{subdivtrigrid} is proved using the notion of a {\em smart (semi-transitive) orientation} introduced in this paper. Finally, in Section~\ref{sec5} we apply our main result to subdivisions of triangular grid graphs having equilateral triangle shape and subdivisions of the Sierpi\'{n}ski gasket graph. 

\section{Definitions, notation, and known results}\label{sec2}

Suppose that $w$ is a word and $x$ and $y$ are two distinct letters in $w$.
We say that $x$ and $y$ {\it alternate} in $w$ if the deletion of all other
letters from the word $w$ results in either $xyxy\cdots$ or $yxyx\cdots$.

A graph $G=(V,E)$ is {\it word-representable} if there exists a word $w$ over
the alphabet $V$ such that letters $x$ and $y$ alternate in $w$ if and only if 
$(x,y)\in E$ for each $x\neq y$. We say that {\it w represents G}, and such a word $w$ 
is called a {\it word-representant} for $G$.
For example, if the word $w=134231241$ then the subword induced with letters $1$ and $2$
is $12121$, hence letters $1$ and $2$ are alternating in $w$, and thus the respective nodes are connected in $G$. On the other hand, the letters $1$ and $3$ 
are not alternating in $w$, because removing all other letters we obtain $1331$; thus, $1$ and $3$ are not connected in $G$. Figure~\ref{wordandgraph} shows the graph represented by $w$.

If each letter appears exactly $k$ times in a word-representant of a graph, the word is {\it k-uniform} and
the graph is said to be {\it k-word-representable}. For example, the word $w'=13423124$ is also a word-representant 
for the graph shown in Figure~\ref{wordandgraph}, so the graph is 2-word-representable. The following theorem establishes equivalence of the notions of word-representability and uniform word-representability.

\begin{thm}\label{unif-repres-equival}(\cite{Kitaev08a}) A graph $G$ is word-representable if and only if there exists $k$ such that $G$ is $k$-word-representable. \end{thm}

Next, we define key objects of our interest including {\em semi-transitive orientations}, {\em triangular grid graphs} and the {\em Sierpi\'{n}ski gasket graph}. For graph-theoretic terminology not defined in this paper, the reader is referred to \cite{Bondy76}.

\subsection{Semi-transitive orientations}

A directed graph (digraph) is {\it semi-transitive} if it is acyclic, and
for any directed path $v_1\rightarrow v_2\rightarrow\cdots\rightarrow v_k$ with $v_i\in V$
for all $i, 1\leq i \leq k$, either
\begin{enumerate}
 \item[$\bullet$] there is no edge $v_1\rightarrow v_k$, or
 \item[$\bullet$] the edge $v_1\rightarrow v_k$ is present and there are edges $v_i\rightarrow v_j$ for all 
     $1\leq i< j \leq k$. That is, in this case, the (acyclic) subgraph induced by the vertices
     $v_1,\ldots,v_k$ is transitive. 
\end{enumerate}
We call such an orientation a {\em semi-transitive orientation}.

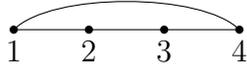
\begin{figure}[!htbp]
 \begin{center}
\begin{tikzpicture}
 \draw (0,0) node[below]{1} -- ++(1,0)node[below]{2} -- ++(1,0)node[below]{3} -- ++(1,0)node[below]{4};
 \draw (0,0) .. controls (0.5,.5)and (2.5,.5).. (3,0);
 \fill[black!100] (0,0) circle(0.3ex)
                 ++(1,0) circle(0.3ex)
                 ++(1,0) circle(0.3ex)
                 ++(1,0) circle(0.3ex);
\end{tikzpicture}
\caption{\label{wordandgraph} The graph represented by the word $w=134231241$.}
\end{center}
\end{figure}

We can alternatively define semi-transitive orientations in terms of induced subgraphs. A \emph{semi-cycle} is the directed
acyclic graph obtained by reversing the direction of one arc of a directed cycle. An acyclic digraph is a
\emph{shortcut} if it is induced by the vertices of a semi-cycle and contains a pair of non-adjacent vertices. Thus, a
digraph on the vertex set $\{ v_1, \ldots, v_k\}$ is a shortcut if it contains a directed path $v_1\rightarrow
v_2\rightarrow \cdots \rightarrow v_k$, the arc $v_1\rightarrow v_k$, and it is missing an arc $v_i\rightarrow v_j$
for some $1 \le i < j \le k$; in particular, we must have $k\geq 4$, so that any shortcut is on at least four
vertices. Slightly abusing the terminology, in this paper we refer to the arc  $v_1\rightarrow v_k$ in the last definition as a shortcut (a more appropriate name for this would be {\em shortcut arc}). Figure~\ref{shortcut} gives examples of shortcuts, where the edges $1\rightarrow 4$, $2\rightarrow 5$ and $3\rightarrow 6$ are missing, and hence $1\rightarrow 5$, $1\rightarrow 6$ and $2\rightarrow 6$ are shortcuts.

Thus, an orientation of a graph is semi-transitive if it is acyclic and contains no shortcuts.

\begin{figure}[!htbp]
 \begin{center}
\begin{tikzpicture}
 \draw (0,0) node[below]{1} -- ++(1,0)node[below]{2} -- ++(1,0)node[below]{3} -- ++(1,0)node[below]{4} 
             -- ++(1,0)node[below]{5} -- ++(1,0)node[below]{6};
 \foreach \i in {0,1,...,4}
 { \draw[,directed] (\i,0)--+(1,0);
 }
 \draw[,directed] (0,0) .. controls (0.5,.5)and (1.5,.5).. (2,0);
  \draw[,directed] (0,0) .. controls (0.5,-.5)and (3.5,-.5).. (4,0);
   \draw[,directed] (0,0) .. controls (0.5,.85)and (4.5,.85).. (5,0);
    \draw[,directed] (1,0) .. controls (1.5,.5)and (2.5,.5).. (3,0);
     \draw[,directed] (1,0) .. controls (1.5,-.75)and (4.5,-.75).. (5,0);
      \draw[,directed] (2,0) .. controls (2.5,.5)and (3.5,.5).. (4,0);
       \draw[,directed] (3,0) .. controls (3.5,.5)and (4.5,.5).. (5,0);
 \fill[black!100] (0,0) circle(0.3ex)
                 ++(1,0) circle(0.3ex)
                 ++(1,0) circle(0.3ex)
                 ++(1,0) circle(0.3ex)
                 ++(1,0) circle(0.3ex)
                 ++(1,0) circle(0.3ex);
\end{tikzpicture}

\caption{ An example of a shortcut \label{shortcut}}
\end{center}
\end{figure}
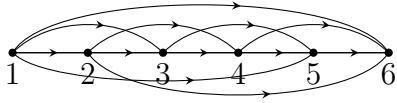

Halld\'{o}rsson et al.\ \cite{Halldorsson15} proved the following theorem that characterizes  word-representable graphs in terms of graph orientations.
\begin{thm}\label{semitra}
 (\cite{Halldorsson11}) A graph is word-representable if and only if it admits a semi-transitive orientation.
\end{thm}

An immediate corollary to Theorem~\ref{semitra} is the following result.

\begin{thm}{\label{color}}(\cite{Halldorsson11}) $3$-colorable graphs are word-representable.
\end{thm}

\subsection{Triangular grid graphs}


The infinite graph $T^{\infty}$ associated with the two-dimensional {\em triangular grid} (also known as the {\em triangular tiling
graph}, see \cite{Ren87} and \cite{Gordon08}) is a graph drawn in the plane with straight-line edges and defined as follows.

The vertices of $T^{\infty}$
are represented by a linear combination $xp + yq$ of the two vectors $p = (1, 0)$ and $q = (1/2,{\sqrt{3}}/{2}) $ with integers
$x$ and $y$. Thus, we may identify the vertices of $T^{\infty}$ with
pairs $(x, y)$ of integers, and thereby the vertices of $T^{\infty}$ are
points with Cartesian coordinates $(x + {y}/{2}, y \sqrt{3}/{2})$ .
Two vertices of $T$ are adjacent if and only if the Euclidean
distance between them is equal to 1 (see Figure \ref{Tinfty}). A line $\ell$ containing an edge of $T^{\infty}$ is called a {\it grid line}. Note that the degree of any vertex of $T$ is equal to 6. We refer to the triangular faces of $T^{\infty}$ as {\em cells}.

A {\it triangular grid graph} is a finite subgraph of $T^{\infty}$, which is formed by all edges bounding finitely many cells. Note that in our definition, a triangular grid graph does not have to be an induced subgraph of $T^{\infty}$. It is easy to see that  $T^{\infty}$ is 3-colorable, and thus any triangular grid graph is 3-colorable. Therefore, triangular grid graphs are word-representable by Theorem \ref{color}.


\begin{figure}[!htbp]
\begin{center}
\begin{tikzpicture}[scale=0.5]
\clip (0.3,0.3) rectangle (7.7,5.5);
\foreach \x in {-10,...,10} 
 \foreach \y in {-10,...,10}
 { \fill[black!100] (\x + .5* \y, 0.866*\y) circle(0.5ex);
 }
\foreach \x in {-10,...,10} 
{ \draw (\x,0)-- +(10,17.32);
  \draw (\x,0)-- +(-10,17.32);
}
\foreach \y in {-10,...,10} 
{ \draw (-1,0.866*\y)-- +(10,0);

}
\end{tikzpicture}
\caption{\label{Tinfty} A fragment of the graph  $T^{\infty}$. }
\end{center}
\end{figure}
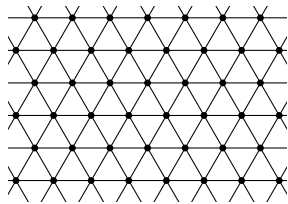

The operation of {\it subdivision of a cell} is putting a new vertex inside the cell and making it to be adjacent to every vertex of the cell. Equivalently, subdivision of a cell is replacing the cell (which is the complete graph $K_3$) by a plane version of the complete graph $K_4$. 
A {\it subdivision of a set $S$ of cells} of a triangular grid graph $G$ is a graph obtained from $G$ by subdividing each cell in $S$. The set $S$ of subdivided cells is called a {\it subdivided set}. For example, Figure~\ref{subdiv} shows $K_4$, the subdivision of a cell, and $A'$, a subdivision of $A$.

If a subdivision of $G$ results in a word-representable graph, then the subdivision is called a {\em word-representable subdivision}. Also, we say that a word-representable subdivision of a triangular grid graph $G$ is {\em maximal}
if subdividing any other cell results in a non-word-representable graph. 
\newsavebox{\Triangle}
\savebox{\Triangle}
{\begin{tikzpicture}[scale=0.5]
  \draw(0,0)--++(0.5,0.866)-- ++(0.5,-0.866)-- ++(-1,0);
  \fill[black!100] (0,0) circle(0.3ex)
                 ++(0.5,0.866)circle(0.3ex)
                 ++(0.5,-0.866)circle(0.3ex);
 \end{tikzpicture}
}
\newsavebox{\Aa}
\savebox{\Aa}
{\begin{tikzpicture}[scale=0.5]
  \draw (0,0)-- ++(0.5,0.866)-- ++(0.5,0.866)-- ++(0.5,-0.866)-- ++(0.5,-0.866)-- ++(-1,0) 
             -- ++(0.5,0.866)-- ++(-1,0)-- ++(0.5,-0.866)-- ++(-1,0);
  \fill[black!100] (0,0) circle(0.3ex)
                ++(0.5,0.866) circle(0.4ex) 
                ++(0.5,0.866) circle(0.4ex) 
                ++(0.5,-0.866) circle(0.4ex) 
                ++(0.5,-0.866) circle(0.4ex) 
                ++(-1,0) circle(0.3ex) ;
 \end{tikzpicture}
}

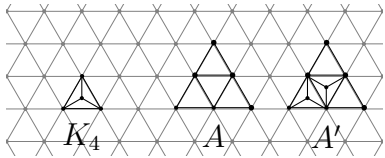
\begin{figure}[!htbp]
 \begin{center}
\begin{tikzpicture}[scale=0.5]
\clip (0.5,0.5) rectangle (10.5,4.5);
\foreach \x in {-10,...,15} 
 \foreach \y in {-10,...,15}
 { \fill[gray!100] (\x + .5* \y, 0.866*\y) circle(0.3ex);
 }
\foreach \x in {-10,...,15} 
{ \draw[gray,very thin] (\x,0)-- +(10,17.32);
  \draw[gray,very thin] (\x,0)-- +(-10,17.32);
}
\foreach \y in {-10,...,15} 
{ \draw[gray,very thin] (-1,0.866*\y)-- +(15,0);
}

\draw (2,1.732)  node[anchor=south west,inner sep=-0.3pt,outer sep=-0.3pt] {\usebox{\Triangle}}
      --++(0.5,0.2887)
      --++(0.5,-0.2887)
      ++(-0.5,0.2887)--+(0,0.5774);
\fill[black!100] (2,1.732)+(0.5,0.2887) circle(0.3ex);
\draw (5,1.732)  node[anchor=south west,inner sep=-0.3pt,outer sep=-0.3pt] {\usebox{\Aa}}
      +(3,0)  node[anchor=south west,inner sep=-0.3pt,outer sep=-0.3pt]  {\usebox{\Aa}};
\draw (8,1.732)--++(0.5,0.2887)--++(0.5,-0.2887)++(-0.5,0.2887)--++(0,0.5774)
      --++(0.5,-0.2887)--++(0.5,0.2887)++(-0.5,-0.2887)--++(0,-0.5774);
\fill[black!100] (8,1.732)+(0.5,0.2887) circle(0.3ex)
                 +(1,0.5774) circle(0.3ex);
\path (2.5,1) node {$K_4$}
      ++(3.5,0) node {$A$}
      ++(3,0) node {$A'$};
\end{tikzpicture}
\caption{\label{subdiv} Examples of subdivisions: $K_4$ is the subdivision of a cell, and $A'$ is a subdivision of $A$.}
\end{center}
\end{figure}


An edge of a triangular grid graph $G$ shared with a cell in $T^{\infty}$ that does not belong to $G$ is called a {\em boundary edge}. Recall that a cell belongs to $G$ if and only if all of its edges belong to $G$. A non-boundary edge belonging to $G$ is called an {\em interior edge}. A cell in $G$ that is incident to at least one boundary edge is called a {\em boundary cell}. A non-boundary cell in $G$ is called an {\em interior cell}. The boundary edges in the graphs $H$ and $K$ in Figure~\ref{BoundaryEdge} are in bold. 

 \newsavebox{\Abig}
 \sbox{\Abig}
 {\begin{tikzpicture}[scale=1]
  \draw[thin](0,0)--++(0.5,0.866)-- ++(0.5,-0.866)-- ++(-1,0);
  \fill[black!100] (0,0) circle(0.3ex)
                 ++(0.5,0.866)circle(0.3ex)
                 ++(0.5,-0.866)circle(0.3ex);
 \end{tikzpicture}
 }
 \newsavebox{\Vbig}
 \sbox{\Vbig}
 {\begin{tikzpicture}[scale=1]
  \draw[thin](0,0)--++(0.5,0.866)-- ++(-1,0)-- ++(0.5,-0.866);
  \fill[black!100] (0,0) circle(0.3ex)
                 ++(0.5,0.866)circle(0.3ex)
                 ++(-1,0)circle(0.3ex);
 \end{tikzpicture}
 }

\begin{figure}[!htbp]
 \begin{center}
\begin{tikzpicture}
\clip (-0.5,0.3) rectangle (7.5,4);
\foreach \x in {-10,...,10} 
 \foreach \y in {-10,...,10}
 { \fill[gray!100] (\x + .5* \y, 0.866*\y) circle(0.3ex);
 }
\foreach \x in {-10,...,10} 
{ \draw[gray,very thin] (\x,0)-- +(10,17.32);
  \draw[gray,very thin] (\x,0)-- +(-10,17.32);
}
\foreach \y in {-10,...,10} 
{ \draw[gray,very thin] (-1,0.866*\y)-- +(10,0);
}
\draw (1,1.732)  node[anchor=south west,inner sep=-0.6pt,outer sep=-0.6pt] {\usebox{\Abig}}
      ++(0.5,0)  node[anchor=south west,inner sep=-0.6pt,outer sep=-0.6pt] {\usebox{\Vbig}}
      ++(0,0.866)  node[anchor=south west,inner sep=-0.6pt,outer sep=-0.6pt] {\usebox{\Abig}}
      ++(-0.5,0)  node[anchor=south west,inner sep=-0.6pt,outer sep=-0.6pt] {\usebox{\Vbig}}
      ++(-0.5,0)  node[anchor=south west,inner sep=-0.6pt,outer sep=-0.6pt] {\usebox{\Abig}}
      ++(0,-0.866)  node[anchor=south west,inner sep=-0.6pt,outer sep=-0.6pt] {\usebox{\Vbig}};
\draw[very thick] (1,1.732) node[left]{$1$}--++(1,0)node[right]{$2$}--++(0.5,0.866)node[right]{$3$}--++(-0.5,0.866) node[right]{$4$}
               --++(-1,0)node[left]{$5$}--++(-0.5,-0.866)node[left]{$6$}--++(0.5,-0.866)+(0.3,0.7) node{$7$};

\draw (3+0.5,1.732)  node[anchor=south west,inner sep=-0.6pt,outer sep=-0.6pt] {\usebox{\Vbig}}
      ++(0.5,-0.866)  node[anchor=south west,inner sep=-0.6pt,outer sep=-0.6pt] {\usebox{\Vbig}}
      ++(1,0.866)  node[anchor=south west,inner sep=-0.6pt,outer sep=-0.6pt] {\usebox{\Abig}}
      ++(-0.5,0.866)  node[anchor=south west,inner sep=-0.6pt,outer sep=-0.6pt] {\usebox{\Abig}}
      ++(-1,0)  node[anchor=south west,inner sep=-0.6pt,outer sep=-0.6pt] {\usebox{\Abig}}
      ++(0,-0.866)  node[anchor=south west,inner sep=-0.6pt,outer sep=-0.6pt] {\usebox{\Vbig}};
      
\draw[very thick] (4,1.732) node[left]{$1$}--++(0.5,-0.866)node[left]{$2$}--++(0.5,0.866)node[below]{$3$}--++(1,0)node[right]{$4$}
                 --++(-0.5,0.866)node[right]{$5$}--++(-0.5,0.866)node[right]{$6$}
                 --++(-0.5,-0.866) node[below]{$7$}--++(-0.5,0.866)node[right]{$8$}--++(-0.5,-0.866) node[left]{$9$}
                 --++(0.5,-0.866);
\draw[very thick] (4,1.732)--++(1,0)--++(0.5,0.866)--++(-1,0)--++(-0.5,-0.866);
\draw(1.5,1.2) node[below] {$H$} +(3.7,0) node[below] {$K$} ;

\end{tikzpicture}

\caption{\label{BoundaryEdge} Graphs $H$ and $K$, where boundary edges are in bold.}
\end{center}
\end{figure}
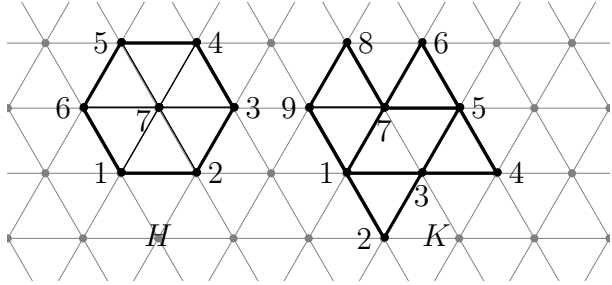

A subdivision of a triangular grid graph that involves subdivision of just boundary cells is called a {\it boundary subdivision}. A boundary edge parallel to the edge $(1,2)$ (resp., $(2,3)$, $(3,4)$, $(4,5)$, $(5,6)$ and $(1,6)$) in the graph $H$ in Figure~\ref{BoundaryEdge} and having the same layout of the boundary cell incident to it, is of type S (resp., SE, NE, N, NW and SW), which stands for ``South'' (resp., ``South-East'', ``North-East'', ``North'', ``North-West'' and ``South-West''). For example, the boundary edges of the graph $K$ in Figure~\ref{BoundaryEdge}
$(1,2)$ and $(1,9)$ are  of type SW, while the boundary edges $(3,5)$, $(7,6)$ and $(8,9)$ are of type NW. 
The {\em property set} of a boundary cell is the set of types of boundary edges incident to the cell. For example, in the graph $K$ in Figure \ref{BoundaryEdge}, the types of the cells 123 and 179 are \{N, SW, SE\} and \{SW, SE\}, respectively.

\subsection{Sierpi\'{n}ski gasket graph}

The two-dimensional {\em Sierpi\'{n}ski gasket graph}, a lattice version of the {\em Sierpi\'{n}ski gasket}, also known as {\em Sierpi\'{n}ski triangle}, is one of 
the most studied self-similar fractal-like graphs.
The construction of this graph, 
denoted by $SG(n)$ for initial stages is shown in Figure \ref{SGn}. 
At stage $n = 0$,
it is an equilateral triangle, i.e. a cell in $T^{\infty}$,
while a stage $n + 1$ giving $SG(n+1)$ is obtained by the juxtaposition of three graphs $SG(n)$ constructed on stage $n$. 
It is not difficult
to see that $SG(n)$ is a triangular grid graph and it has $\frac{3}{2}(3^n+1)$
vertices and $3^{n+1}$ edges.

\newsavebox{\A}
\savebox{\A}
{\begin{tikzpicture}[scale=0.5]

\draw (0,0) -- ++(0.5774,1) --++(0.5774,1)
                   --++(0.5774,-1) --++(0.5774,-1)
                   --++(-2*0.5774,0)--++(-0.5774,1)--++(2*0.5774,0)
                   --++(-0.5774,-1)--+(-2*0.5774,0)
                   (2*0.5774,0);
\fill[black!100] (0,0) circle(0.5ex)
                 ++(0.5774,1) circle(0.5ex)
                 ++(0.5774,1) circle(0.5ex)
                 ++(0.5774,-1) circle(0.5ex)
                 ++(0.5774,-1) circle(0.5ex)
                 ++(-2*0.5774,0)circle(0.5ex)
               ;
\end{tikzpicture}
}

\begin{figure}[!htbp]
\begin{center}
\begin{tikzpicture}[scale=0.5]
 \draw (0,0) -- ++(0.5774,1) --++(0.5774,-1)
                   --++(-2*0.5774,0);
 \fill[black!100] (0,0) circle(0.3ex)
                 ++(0.5774,1) circle(0.5ex)
                 +(0.5774,-1) circle(0.5ex);
 \path (4,1) node {\usebox{\A}};
  
 \path (0,0)++(8,1) node {\usebox{\A}}
      +(0.5774*2,2) node {\usebox{\A}}
      +(0.5774*4,0) node {\usebox{\A}}
      ;
 \path (0,0)++(14.5,1) node {\usebox{\A}}
      +(0.5774*2,2) node {\usebox{\A}}
      +(0.5774*4,0) node {\usebox{\A}}
      ;
 \path (0.5774*4,4)++(14.5,1) node {\usebox{\A}}
      +(0.5774*2,2) node {\usebox{\A}}
      +(0.5774*4,0) node {\usebox{\A}}
      ; 
 \path (0.5774*8,0)++(14.5,1) node {\usebox{\A}}
      +(0.5774*2,2) node {\usebox{\A}}
      +(0.5774*4,0) node {\usebox{\A}}
      ;      
 \path (0.5,-0.7) node {$SG(0)$}
       ++(3.5,0) node {$SG(1)$}
       ++(5,0) node {$SG(2)$}
       ++(8.5,0) node {$SG(3)$};
\end{tikzpicture}
\caption{\label{SGn} The first four stages corresponding to $n = 0, 1, 2, 3$ in construction of the two-dimensional Sierpi\'{n}ski gasket graph $SG(n)$.}
\end{center}
\end{figure}
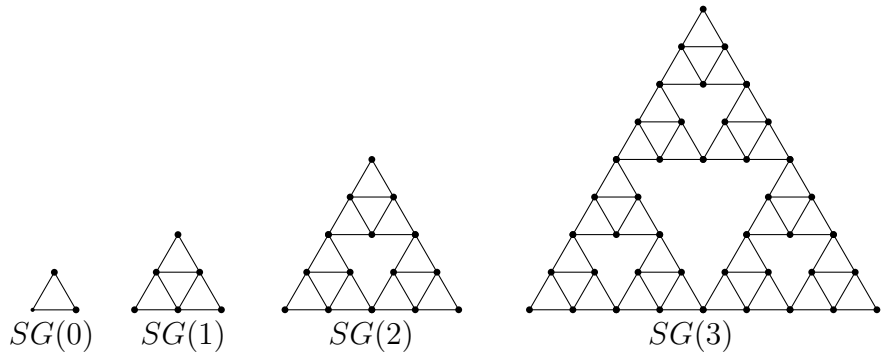

\section{The minimal non-word-representable subdivisions of a triangular grid graph}\label{sec3}

The following theorem can be proved by showing that the graph in question does not accept a semi-transitive orientation, which requires considering several cases and subcases, and then applying Theorem~\ref{semitra}. However, we provide a combinatorics on words type of proof that is essentially the same as proving in  \cite{Halldorsson10} that the graph co-($T_2)$ is non-word-representable.

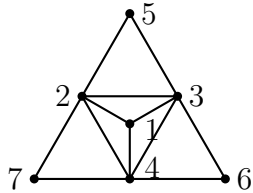
\begin{figure}[!htbp]
\begin{center}
\begin{tikzpicture}[scale=1.1]
\draw[thick] (0,0) node[anchor=east]{7}-- ++(0.5774,1) node[anchor=east] {2}--++(0.5774,1) node[anchor=west] {5}
                   --++(0.5774,-1) node[anchor=west] {3}--++(0.5774,-1) node[anchor=west] {6}
                   --++(-2*0.5774,0)--++(-0.5774,1)--++(2*0.5774,0)
                   --++(-0.5774,-1)--+(-2*0.5774,0)
                   (2*0.5774,0)--++(0,2/3) --+(-0.5774,1/3)
                   (2*0.5774,2/3)--+(0.5774,1/3) ;
\path (1.2,0.14) node [anchor=west]{4}
     +(0,0.45) node [anchor= west]{1};

\fill[black!100] (0,0) circle(0.3ex)
                 ++(0.5774,1) circle(0.3ex)
                 ++(0.5774,1) circle(0.3ex)
                 ++(0.5774,-1) circle(0.3ex)
                 ++(0.5774,-1) circle(0.3ex)
                   (2*0.5774,0) circle(0.3ex)
                 ++(0,2/3) circle(0.3ex);
\end{tikzpicture}
\caption{\label{nonRepTri} A non-word-representable subdivision $A''$ of the triangular grid graph $A$ shown in Figure~\ref{subdiv}.}
\end{center}
\end{figure}

\begin{thm}\label{A''-non-word-repr}
 The graph $A''$ in Figure~\ref{nonRepTri} is non-word-representable.
\end{thm}

\begin{proof}
 Assume that the graph $A''$ is word-representable. Thus, by Theorem~\ref{unif-repres-equival}, it is $k$-word-representable for some $k$. Let $w$ be a $k$-uniform word representing $A''$, and let $x^i$ denote the $i$-th occurrence (from left to right) of a letter $x$ in $w$. 
 
 The vertices $1,2,3,4$ form a clique; so their appearances $1^i,2^i,3^i,4^i$ in W must be in the same
order for each $i=1,2,\ldots,k$. By symmetry and taking into account that a cyclic shift of $w$ represents the same graph (see~\cite{Kitaev08a}), without loss of generality, we may assume
that the order of occurrences of the four letters is $1234$.
Now let $I_i$ be the $[2^i,4^i]$-interval in $w$ for $i=1,2,\ldots,k$, which is all letters in $w$ between $2^i$ and $4^i$. Then there are two possible cases.
\begin{enumerate}
 \item[i.] $7\in I_j$ for some $j\in\{1,2,\cdots,k\}$. Since $\{2,4,7\}$ form a clique, 7 must be inside each
of the intervals $I_1,I_2,\ldots,I_k$. But then 1 and 7 are alternating in $w$, hence they are adjacent in $A''$, a
contradiction.
 \item[ii.] $7\notin I_j$ for every $j=1,2,\ldots,k$. Again, since
 7 is adjacent to both 2 and 4, each pair of consecutive intervals $I_j, I_{j+1}$ must be separated by a single 7.
 But then 7 is adjacent to 3, a contradiction.
\end{enumerate}
We are done.\end{proof}

\begin{remark} Note that the nodes $5$ and $6$ do not appear in the proof of Theorem~\ref{A''-non-word-repr}, which may look as exactly the same arguments would prove that the graph $G'$ obtained by removing nodes $5$ and $6$ is non-word-representable, which is not the case. In fact, our proof of Theorem~\ref{A''-non-word-repr} cannot be applied to $G'$ since this graph has no symmetry, which was used by us. 
\end{remark}

\section{Subdivisions of triangular grid graphs of general shape}\label{sec4}

By Theorem~\ref{A''-non-word-repr}, taking into account the hereditary nature of word-representable graphs, if $A''$ is an induced subgraph of a graph $G$ then $G$ is non-word-representable. The following theorem shows that the presence of $A''$ as an induced subgraph is necessary and sufficient condition for a subdivision of a triangular grid graph to be non-word-representable. Thus, $A''$ is the minimal non-word-representable subdivision for all triangular grid graphs, i.e. it is  an induced subgraph of any non-word-representable subdivision of a triangular grid graph. On the other hand, there is the unique maximal word-representable subdivision for any triangular grid graph $G$, which we call the {\em maximum subdivision} of $G$. This subdivision is obtained by subdividing {\em all} boundary cells of $G$.
  
Our main result in this paper is the following theorem.

\begin{thm}\label{subdivtrigrid}
A subdivision of a triangular grid graph $G$ is word-representable if and only if it has no induced subgraph
 isomorphic to $A''$, that is, $G$ has no subdivided interior cell.
\end{thm}

Proving the following proposition gave us an idea how to prove Theorem~\ref{subdivtrigrid} via a special type of orientations.

 \begin{prop}\label{prop-H}
 The graph obtained by subdividing all the cells of $H$ (shown in Figure \ref{BoundaryEdge}) is word-representable.
 \end{prop}
\begin{proof}
It is sufficient to show that the orientation of $H$
shown in Figure \ref{Hori} is semi-transitive, which can be done by direct inspection checking that no arc can be a shortcut or be involved in a directed cycle.
\end{proof}

\newsavebox{\AsixOri}
\newsavebox{\SixOri}
\sbox{\AsixOri}
{
\begin{tikzpicture}  
  \draw[,reverse directed] (0,0) -- +(0.5774,1);  
  \draw[,directed] (0,0) -- +(1.1547,0);
  \draw[,directed] (0,0) -- +(0.5774,-1);
  \draw[,directed] (0,0) -- +(-0.5774,-1);
  \draw[,reverse directed] (0,0) -- +(-1.1547,0);
  \draw[,reverse directed] (0,0) -- +(-0.5774,1);
  
  \draw[,reverse directed] (0.5774,1) -- +(-1.1547,0);
  \draw[,directed] (0.5774,1) -- +(0.5774,-1);
  \draw[,reverse directed] (0.5774,-1) -- +(0.5774,1);;
  \draw[,reverse directed] (0.5774,-1) -- +(-1.1547,0);
  \draw[,directed] (-1.1547,0) -- +(0.5774,-1);
  \draw[,reverse directed] (-1.1547,0) -- +(0.5774,1); 
  
  \fill[black!100] (0,0) circle(0.3ex);
  \foreach \rd in {0,...,5}
    {  \fill[black!100] [rotate=\rd *60] (0.5774,1) circle(0.3ex);
       \fill[black!100] [rotate=\rd *60] (0,2/3) circle(0.3ex);
    }
  \draw[,directed] (0,2/3) -- +(0.5774,1/3);  
  \draw[,reverse directed] (0,2/3) -- +(-0.5774,1/3); 
  \draw[,directed] (0,2/3) -- +(0,-2/3); 
  
  \draw[,reverse directed] (0.5774,1/3) -- +(0.5774,-1/3); 
  \draw[,reverse directed] (0.5774,1/3) -- +(-0.5774,-1/3); 
  \draw[,reverse directed] (0.5774,1/3) -- +(0,2/3); 
  
  \draw[,reverse directed] (0.5774,-1/3) -- +(0.5774,1/3); 
  \draw[,reverse directed] (0.5774,-1/3) -- +(-0.5774,1/3); 
  \draw[,directed] (0.5774,-1/3) -- +(0,-2/3); 
  
  \draw[,directed] (0,-2/3) -- +(0.5774,-1/3); 
  \draw[,reverse directed] (0,-2/3) -- +(-0.5774,-1/3); 
  \draw[,reverse directed] (0,-2/3) -- +(0,2/3);  

  \draw[,directed] (-0.5774,-1/3) -- +(0.5774,1/3); 
  \draw[,directed] (-0.5774,-1/3) -- +(-0.5774,1/3); 
  \draw[,directed] (-0.5774,-1/3) -- +(0,-2/3);   
  
  \draw[,directed] (-0.5774,1/3) -- +(0.5774,-1/3); 
  \draw[,directed] (-0.5774,1/3) -- +(-0.5774,-1/3); 
  \draw[,reverse directed] (-0.5774,1/3) -- +(0,2/3);   

  \path (0,1.3) node{\footnotesize$c$}
        [rotate=-60] node{\footnotesize$B$}
        [rotate=-60] node{\footnotesize$a$}
        [rotate=-60] node{\footnotesize$C$}
        [rotate=-60] node{\footnotesize$b$}
        [rotate=-60] node{\footnotesize$A$};       
\end{tikzpicture}
}

\sbox{\SixOri}
{
\begin{tikzpicture}
\foreach \x in {0,2,4}{
 \draw[,directed,xshift=\x cm] (0,1)-- +(-0.5774,-1) ;
 \draw[,directed,xshift=\x cm] (0,1)-- +(0.5774,-1);
 \draw[,directed,xshift=\x cm] (0,1)-- +(0,-2/3); 
 \draw[,directed,xshift=\x cm] (-0.5774,0)--(0.5774,0);
 \fill[black!100] [xshift=\x cm] (0,1) circle(0.3ex)
      (0,1/3) circle(0.3ex)
      (-0.5774,0) circle(0.3ex)
      (0.5774,0) circle(0.3ex);
 
 \draw[,reverse directed,shift={(\x ,1.8)}] (0,0)-- +(-0.5774,1) ; 
 \draw[,reverse directed,shift={(\x ,1.8)}] (0,0)-- +(0.5774,1);
 \draw[,reverse directed,shift={(\x ,1.8)}] (0,0)-- +(0,2/3); 
 \draw[,directed,shift={(\x ,1.8)}] (-0.5774,1)--(0.5774,1);
 \fill[black!100] [shift={(\x ,1.8)}] (0,0) circle(0.3ex)
      (0,2/3) circle(0.3ex)
      (-0.5774,1) circle(0.3ex)
      (0.5774,1) circle(0.3ex);
}
 \draw[,reverse directed] (-0.5774,0)-- +(0.5774,1/3); 
  \draw[,directed] (-0.5774+2,0)-- +(0.5774,1/3);
   \draw[,directed] (-0.5774+4,0)-- +(0.5774,1/3);
 \draw[,reverse directed] (0.5774,0)-- +(-0.5774,1/3);
  \draw[,directed] (0.5774+2,0)-- +(-0.5774,1/3);
   \draw[,reverse directed] (0.5774+4,0)-- +(-0.5774,1/3);
   
 \draw[,reverse directed,shift={(0 ,1.8)}] (0,2/3)-- +(0.5774,1/3);
   \draw[,reverse directed,shift={(0 ,1.8)}] (0,2/3)-- +(-0.5774,1/3);
 \draw[,directed,shift={(2 ,1.8)}] (0,2/3)-- +(0.5774,1/3);
   \draw[,directed,shift={(2 ,1.8)}] (0,2/3)-- +(-0.5774,1/3);
 \draw[,directed,shift={(4 ,1.8)}] (0,2/3)-- +(0.5774,1/3);
   \draw[,reverse directed,shift={(4 ,1.8)}] (0,2/3)-- +(-0.5774,1/3);
 \path (0,-.3) node {\footnotesize$A$}
       ++(2,0) node {\footnotesize$B$}
       ++(2,0) node  {\footnotesize$C$}
       ++(0,1.8) node  {\footnotesize$c$}
       ++(-2,0) node  {\footnotesize$b$}
       ++(-2,0) node  {\footnotesize$a$};
\end{tikzpicture}
}

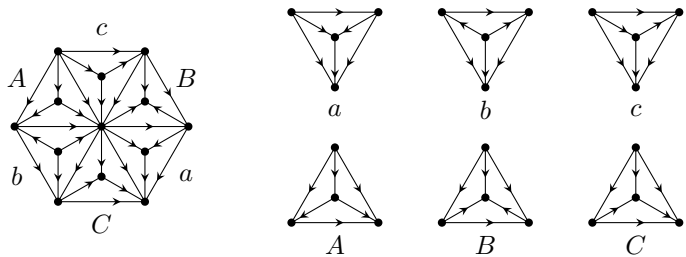
\begin{figure}[!htbp]
\begin{center}
 \begin{picture}(250,90)
\put(0,8){\usebox{\AsixOri}}
\put(110,0){\usebox{\SixOri}}
\end{picture}
\caption{\label{Hori} The semi-transitive orientation of the maximum subdivision of the graph $H$, 
         and six orientations of the subdivided cells.}
\end{center}
\end{figure}

To prove Theorem \ref{subdivtrigrid}, we will describe an orientation of a subdivision of a triangular grid graph
and then prove that this orientation, already used by us in Proposition~\ref{prop-H}, is semi-transitive. First, define a correspondence between the types of boundary
edges and the types of orientations of subdivided cells shown in Figure~\ref{Hori} as following:
NW (respectively, NE, S, SE, SW, N) corresponds to $A$ (respectively, $B$, $C$, $a$, $b$, $c$).

For a subdivision $G'$ of a triangular grid graph $G$ without subdivided interior cell, we direct the edges  of $G'$ as follows:
\begin{enumerate}
 \item Direct all horizontal edges from left to right, and the edges forming $60^\circ$, $90^\circ$, or $120^\circ$ with a horizontal line from top to bottom. We refer to the obtained arcs as {\em grid arcs}.
 \item Direct the edges located inside subdivided cells consistently with an orientation shown in Figure~\ref{Hori} that corresponds to one of the types in the property set of the cell. The obtained non-grid-arcs are called {\em interior arcs}.
\end{enumerate}
If an orientation of a subdivision of $G'$ satisfies the two conditions above, then we say that the orientation is {\it smart}. 

For example, for the subdivision of $K'$ shown in Figure \ref{BounEdgeOri}, the property sets of the subdivided cells $123$ and $567$ are, respectively, 
\{SW, SE\} and \{NW, NE\}. Thus, referring to Figure~\ref{Hori}, the orientations of edges for the cell $123$ can be chosen to be of type $a$ or $b$, while for the cell $567$ of type $A$ or $B$. In particular, the orientation shown in Figure~\ref{BounEdgeOri} is smart, and this orientation can be checked by inspection to be semi-transitive.

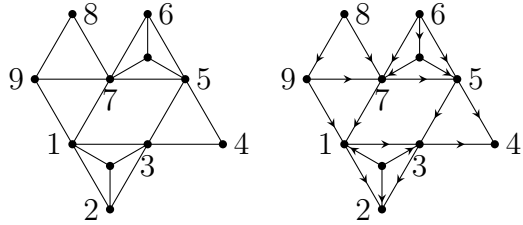
\begin{figure}[!htbp]
 \begin{center}
\begin{tikzpicture}
\draw (0,0)-- +(-0.5,0.866);
\draw (0,0)-- +(0.5,0.866);
\draw (0,0)-- +(1,0);
\draw (0,0)-- +(0.5,-0.2887);
\draw (0,0)-- +(0.5,-0.866);

\draw (-0.5,0.866)-- +(0.5,0.866);
\draw (-0.5,0.866)-- +(1,0);

\draw (0.5,0.866)-- +(-0.5,0.866);
\draw (0.5,0.866)-- +(0.5,0.866);
\draw (0.5,0.866)-- +(1,0);
\draw (0.5,0.866)-- +(0.5,0.2887);

\draw (1,1.732)--+(0,-0.5774);

\draw (1.5,0.866)-- +(-0.5,0.866);
\draw (1.5,0.866)-- +(-0.5,0.2887);
\draw (1.5,0.866)-- +(-0.5,-0.866);
\draw (1.5,0.866)-- +(0.5,-0.866);

\draw (1,0)-- +(1,0);
\draw (1,0)-- +(-0.5,-0.866);
\draw (1,0)-- +(-0.5,-0.2887);

\draw (0.5,-0.866)--+(0,0.5774);
      
\fill[black!100] (0,0) node[left]{$1$} circle(0.3ex)
         +(0.5,-0.2887) circle(0.3ex)
     ++(0.5,-0.866)node[left]{$2$} circle(0.3ex)
     ++(0.5,0.866)node[below]{$3$} circle(0.3ex)
     ++(1,0)node[right]{$4$} circle(0.3ex)
     ++(-0.5,0.866)node[right]{$5$} circle(0.3ex)
     ++(-0.5,0.866)node[right]{$6$} circle(0.3ex)
     ++(-0.5,-0.866) node[below]{$7$} circle(0.3ex)
         +(0.5,0.2887) circle(0.3ex)
     ++(-0.5,0.866)node[right]{$8$} circle(0.3ex)
     ++(-0.5,-0.866) node[left]{$9$} circle(0.3ex);
     
\end{tikzpicture}
\begin{tikzpicture}
\draw[,reverse directed] (0,0)-- +(-0.5,0.866);
\draw[,reverse directed] (0,0)-- +(0.5,0.866);
\draw[,directed] (0,0)-- +(1,0);
\draw[,reverse directed] (0,0)-- +(0.5,-0.2887);
\draw[,directed] (0,0)-- +(0.5,-0.866);

\draw[,reverse directed] (-0.5,0.866)-- +(0.5,0.866);
\draw[,directed] (-0.5,0.866)-- +(1,0);

\draw[,reverse directed] (0.5,0.866)-- +(-0.5,0.866);
\draw[,reverse directed] (0.5,0.866)-- +(0.5,0.866);
\draw[,directed] (0.5,0.866)-- +(1,0);
\draw[,reverse directed] (0.5,0.866)-- +(0.5,0.2887);

\draw[,directed] (1,1.732)--+(0,-0.5774);

\draw[,reverse directed] (1.5,0.866)-- +(-0.5,0.866);
\draw[,reverse directed] (1.5,0.866)-- +(-0.5,0.2887);
\draw[,directed] (1.5,0.866)-- +(-0.5,-0.866);
\draw[,directed] (1.5,0.866)-- +(0.5,-0.866);

\draw[,directed] (1,0)-- +(1,0);
\draw[,directed] (1,0)-- +(-0.5,-0.866);
\draw[,reverse directed] (1,0)-- +(-0.5,-0.2887);

\draw[,reverse directed] (0.5,-0.866)--+(0,0.5774);
      
\fill[black!100] (0,0) node[left]{$1$} circle(0.3ex)
         +(0.5,-0.2887) circle(0.3ex)
     ++(0.5,-0.866)node[left]{$2$} circle(0.3ex)
     ++(0.5,0.866)node[below]{$3$} circle(0.3ex)
     ++(1,0)node[right]{$4$} circle(0.3ex)
     ++(-0.5,0.866)node[right]{$5$} circle(0.3ex)
     ++(-0.5,0.866)node[right]{$6$} circle(0.3ex)
     ++(-0.5,-0.866) node[below]{$7$} circle(0.3ex)
         +(0.5,0.2887) circle(0.3ex)
     ++(-0.5,0.866)node[right]{$8$} circle(0.3ex)
     ++(-0.5,-0.866) node[left]{$9$} circle(0.3ex);
     
\end{tikzpicture}
\caption{\label{BounEdgeOri}  The subdivision $K'$ of the graph $K$ and one of its semi-transitive orientations.}
\end{center}
\end{figure}

Note that a smart orientation may involve eight types of edges, that we call $a_1$, $a_2$, etc, $a_8$; see Figure~\ref{arcs}.

 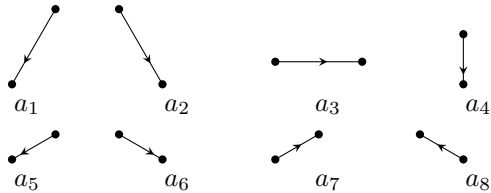
\begin{figure}[!htbp]
 \begin{center}
\begin{tikzpicture}
 \draw[,reverse directed,shift={(0,1)}] (0,0)--+(0.5774,1);
 \draw[,reverse directed,shift={(2,1)}] (0,0)--+(-0.5774,1);
 \draw[,directed,shift={(4-.5,1.3)}] (0,0)-- +(0.5774*2,0);
 \draw[,reverse directed,shift={(6,1)}] (0,0)-- +(0,2/3);
 \draw[,reverse directed,shift={(0,0)}] (0,0)-- +(0.5774,1/3);
 \draw[,reverse directed,shift={(2,0)}] (0,0)-- +(-0.5774,1/3);
 \draw[,directed,shift={(4-.5,0)}] (0,0)-- +(0.5774,1/3);
 \draw[,directed,shift={(6,0)}] (0,0)-- +(-0.5774,1/3);

 \fill[black!100] (0,1) circle(0.3ex)
                  +(0.5774,1)circle(0.3ex)
                  (2,1)circle(0.3ex)
                  +(-0.5774,1)circle(0.3ex)
                  (4-.5,1.3)circle(0.3ex)
                  +(0.5774*2,0)circle(0.3ex)
                  (6,1)circle(0.3ex)
                  +(0,2/3)circle(0.3ex)
                  (0,0)circle(0.3ex)
                  +(0.5774,1/3)circle(0.3ex)
                  (2,0)circle(0.3ex)
                  +(-0.5774,1/3)circle(0.3ex)
                  (4-.5,0)circle(0.3ex)
                  +(0.5774,1/3)circle(0.3ex)
                  (6,0)circle(0.3ex)
                  +(-0.5774,1/3)circle(0.3ex);
 \path   (.2,.7) node{\footnotesize$a_1$}
       ++(2,0) node{\footnotesize$a_2$}
       ++(2,0) node{\footnotesize$a_3$}
       ++(2,0) node{\footnotesize$a_4$}
       ++(0,-1) node{\footnotesize$a_8$}
       ++(-2,0) node{\footnotesize$a_7$}
       ++(-2,0) node{\footnotesize$a_6$}
       ++(-2,0) node{\footnotesize$a_5$};
\end{tikzpicture}
\caption{\label{arcs} Eight types of oriented edges.}
\end{center}
\end{figure}

There are a number of properties that any smart orientation satisfies. Three of these properties, that are easy to see, are listed below. We will be using them, sometimes implicitly by considering fewer subcases, in the proof of Lemma~\ref{LemSmartOri}:   

\begin{itemize}
\item No directed path can connect a vertex on a horizontal line to another vertex to the left of it on the same line.
\item No directed path can get from a horizontal line to a higher horizontal line.
\item The only situation in which a directed path can go down from a horizontal line and return back to it is when the upper interior arcs of type $c$ in Figure~\ref{Hori} are involved.
\end{itemize}

\begin{lem}\label{LemSmartOri-cycle-free}
No smart orientation of the boundary subdivision of a triangular grid graph can have a directed cycle.
\end{lem}
\begin{proof} It is straightforward to see that no directed cycle is possible on just grid arcs, that is, when no interior arc is involved. Thus, if a directed cycle $C$ exists, then it must involve an interior arc.

Further, note that if $e_1$ is an interior arc in $C$, then there must exist an interior arc $e_2$ that is located in the same cell as $e_1$, and $e_1$ and $e_2$ are next to each other in $C$. Without loss the generality, assume that $e_1=x\rightarrow y$ and $e_2=y\rightarrow z$. But then, looking at the six types of orientations of subdivided cells presented in Figure~\ref{Hori}, we see that $e=x\rightarrow y$ is an arc in the oriented graph. Thus, we see that $C'$ obtained from $C$ by removing $e_1$ and $e_2$ and including $e$ is still a directed cycle. Continuing in this manner, we can eliminate all interior arcs and show that there exists a directed cycle containing only grid arcs, which is impossible. 

Thus, $C$ cannot exist, that is, any smart orientation is acyclic.  
\end{proof}

\begin{lem}\label{LemSmartOri}
No smart orientation of the boundary subdivision of a triangular grid graph can contain a shortcut.
\end{lem}
\begin{proof}
In a smart orientation, there are eight types of arcs shown in Figure~\ref{arcs}, and we
will prove that no arc $e=t\rightarrow h$ can be a shortcut. While dealing with smart orientations, sometimes it is convenient to pay attention to coordinates $(x_v,y_v)$ of a vertex $v$ coming from the definition of $T^{\infty}$. The coordinates allow for any two vertices to determine which one of them is to the left of the other one and/or higher than the other one.

\begin{figure}[!htbp]
 \begin{center}
\begin{tikzpicture}[scale=1]
\clip (-1.5,-0.7) rectangle (2.5,1.5);
\foreach \x in {-5,...,5} 
 \foreach \y in {-5,...,5}
 { \fill[gray!100] (\x + .5* \y, 0.866*\y) circle(0.3ex);
 }
\foreach \x in {-5,...,5} 
{ \draw[gray,very thin] (\x,-0.866*2)-- +(10,17.32);
  \draw[gray,very thin] (\x,-0.866*2)-- +(-10,17.32);
}
\foreach \y in {-5,...,5} 
{ \draw[gray,very thin] (-3,0.866*\y)-- +(15,0);
}
\draw[,reverse directed] (0,0)-- +(-0.5,-0.2887);
\draw[,reverse directed] (0,0)-- +(-1,0);
\draw[,reverse directed] (0,0)-- +(-0.5,0.2887);
\draw[,reverse directed] (0,0)-- +(-0.5,0.866);
\draw[,reverse directed] (0,0)-- +(0,0.5774);
\draw[thick,reverse directed] (0,0)-- +(0.5,0.866);
\draw[,reverse directed] (0,0)-- +(0.5,0.2887);
\draw[,reverse directed] (0,0)-- +(0.5,-0.2887);
\draw[,directed] (0.5,0.866)-- +(-0.5,0.2887);
\draw[,directed] (0.5,0.866)-- +(-0.5,-0.2887);
\draw[,directed] (0.5,0.866)-- +(0.5,0.2887);
\draw[,directed] (0.5,0.866)-- +(1,0);
\draw[,directed] (0.5,0.866)-- +(0.5,-0.2887);
\draw[,directed] (0.5,0.866)-- +(0.5,-0.866);
\draw[,directed] (0.5,0.866)-- +(0,-0.5774);
\fill[black!100] (0,0)   node[below]{1} circle(0.3ex)
                 +(-0.5,-0.2887) node[left]{3} circle(0.3ex)
                 +(-1,0)node[left]{4} circle(0.3ex)
                 +(-0.5,0.2887)node[left]{5} circle(0.3ex)
                 +(-0.5,0.866)node[left]{6} circle(0.3ex)
                 +(0,0.5774)node[above]{7} circle(0.3ex)
                 +(0.5,0.2887)node[right]{8} circle(0.3ex)
                 +(0.5,-0.2887)node[right]{9} circle(0.3ex)
                 ++(0.5,0.866)node[above]{2} circle(0.3ex)
                    +(-0.5,0.2887) node[left]{10} circle(0.3ex)
                    +(0.5,0.2887) node[right]{11} circle(0.3ex)
                    +(1,0) node[right]{12} circle(0.3ex)
                    +(0.5,-0.2887) node[right]{13} circle(0.3ex)
                    +(0.5,-0.866) node[right]{14} circle(0.3ex)
                    ;          
\end{tikzpicture}
\caption{\label{arc1} The case when the arc $e=2\rightarrow 1$ is of type $a_1$.}
\end{center}
\end{figure}
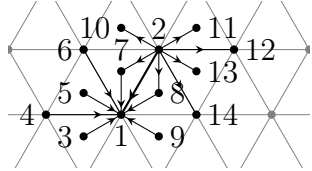

\begin{enumerate}
 \item[Case 1.] Suppose that the arc $e=2\rightarrow 1$ is of type $a_1$, as shown in Figure \ref{arc1}. 
 If $a_1$ is a shortcut, then there exists a directed path $P$ of length at least 3 from $2$ to $1$ 
 (this path does not involve $e$). Suppose that $P$ ends with $e'$. Then $e'=m\rightarrow 1$ can possibly
 be $3\rightarrow1$, $4\rightarrow1$, $5\rightarrow1$, $6\rightarrow1$, $7\rightarrow1$, $8\rightarrow1$, or $9\rightarrow1$ 
 (of type $a_7$, $a_3$, $a_6$, $a_2$, $a_5$, or $a_8$, respectively). 
 However, $e'$ cannot be $3\rightarrow1$, $4\rightarrow1$, $5\rightarrow1$, $6\rightarrow 1$ or $7\rightarrow 1$, because
 in each of these cases $x_m<x_2$ forcing $P$ to begin with an arc $e''=2\rightarrow s$ of type $a_5$ or $a_8$
 ($2\rightarrow7$ or $2\rightarrow10$ in Figure \ref{arc1}), which is impossible by the following reasons. 
 The arc $2\rightarrow7$ is in orientation of type $a$ forcing $P$ be of length 2, contradiction,
 and the arc $2\rightarrow10$ is in orientation of type $B$, so that $10$ would be a sink, contradiction. 
 On the other hand, $e'$ is 
 not $8\rightarrow1$ or $9\rightarrow1$, since the arc coming to
 $8$ must be of type $a_4$ forcing $P$ be of length 2, contradiction, while the arc $9\rightarrow1$ is in 
 orientation of type $b$, so that $9$ would be a source, contradiction.
 Thus, $e$ is not a shortcut in this case. 

 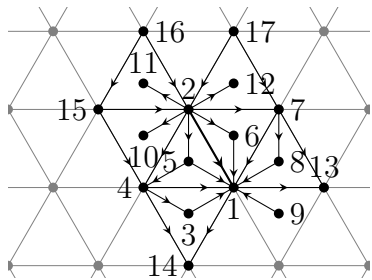
\begin{figure}[!htbp]
 \begin{center}
\begin{tikzpicture}[scale=1.2]
\clip (-2.5,-1) rectangle (1.5,2);
\foreach \x in {-5,...,5} 
 \foreach \y in {-5,...,5}
 { \fill[gray!100] (\x + .5* \y, 0.866*\y) circle(0.3ex);
 }
\foreach \x in {-5,...,5} 
{ \draw[gray,very thin] (\x,-0.866*2)-- +(10,17.32);
  \draw[gray,very thin] (\x,-0.866*2)-- +(-10,17.32);
}
\foreach \y in {-5,...,5} 
{ \draw[gray,very thin] (-3,0.866*\y)-- +(15,0);
}
\draw[,reverse directed] (0,0)-- +(-0.5,-0.2887);
\draw[,reverse directed] (0,0)-- +(-1,0);
\draw[,reverse directed] (0,0)-- +(-0.5,0.2887);
\draw[thick,reverse directed] (0,0)-- +(-0.5,0.866);
\draw[,reverse directed] (0,0)-- +(0,0.5774);
\draw[,reverse directed] (0,0)-- +(0.5,0.866);
\draw[,reverse directed] (0,0)-- +(0.5,0.2887);
\draw[,reverse directed] (0,0)-- +(0.5,-0.2887);

\draw[,directed] (0,0)-- +(-0.5,-0.866);
\draw[,directed] (0,0)-- +(1,0);

\draw[,directed] (-0.5,0.866)-- +(-0.5,0.2887);
\draw[,directed] (-0.5,0.866)-- +(-0.5,-0.2887);
\draw[,directed] (-0.5,0.866)-- +(0.5,0.2887);
\draw[,directed] (-0.5,0.866)-- +(1,0);
\draw[,directed] (-0.5,0.866)-- +(0.5,-0.2887);
\draw[,directed] (-0.5,0.866)-- +(0,-0.5774);
\draw[,directed] (-0.5,0.866)-- +(-0.5,-0.866);
\draw[,reverse directed] (-0.5,0.866)-- +(-0.5,0.866);
\draw[,reverse directed] (-0.5,0.866)-- +(0.5,0.866);
\draw[,reverse directed] (-0.5,0.866)-- +(-1,0);

\draw[,reverse directed] (-1,0)-- +(-0.5,0.866);
\draw[,directed] (-1,0)-- +(0.5,-0.866);
\draw[,reverse directed] (-1,0)-- +(0.5,0.2887);
\draw[,directed] (-1,0)-- +(0.5,-0.2887);
\draw[,reverse directed] (-1.5,0.866)-- +(0.5,0.866);

\draw[,reverse directed] (0.5,0.866)-- +(-0.5,0.866);
\draw[,directed] (0.5,0.866)-- +(0.5,-0.866);
\draw[,directed] (0.5,0.866)-- +(0,-0.5774);
\fill[black!100] (0,0)   node[below]{1} circle(0.3ex)
                  (-0.5,-0.2887) node[below]{3} circle(0.3ex)
                  (-1,0)node[left]{4} circle(0.3ex)
                  (-0.5,0.2887)node[left]{5} circle(0.3ex)
                  (0,0.5774)node[right]{6} circle(0.3ex)
                  (0.5,0.866)node[right]{7} circle(0.3ex)
                  (0.5,0.2887)node[right]{8} circle(0.3ex)
                  (0.5,-0.2887)node[right]{9} circle(0.3ex)
                  (-0.5,-0.866)node[left]{14} circle(0.3ex)
                  (1,0)node[above]{13} circle(0.3ex)
                 (-0.5,0.866)node[above]{2} circle(0.3ex)
                    +(-0.5,-0.2887) node[below]{10} circle(0.3ex)
                    +(-0.5,0.2887) node[above]{11} circle(0.3ex)
                    +(0.5,0.2887) node[right]{12} circle(0.3ex)
                    +(-0.5,0.866)node[right]{16} circle(0.3ex)
                    +(0.5,0.866)node[right]{17} circle(0.3ex)
                    +(-1,0)node[left]{15} circle(0.3ex)
                    ;          
\end{tikzpicture}
\caption{\label{arc2} The case when the arc $e=2\rightarrow 1$ is of type $a_2$.}
\end{center}
\end{figure}
 \item[Case 2.] Suppose that the arc $e=2\rightarrow 1$ is of type $a_2$, shown in Figure \ref{arc2}. 
 If $a_2$ is a shortcut, then there exists a directed path $P$ of length at least 3 from $2$ to $1$ 
 (this path does not involve $e$). Suppose that $P$ begins with $e'$. Then $e'=2\rightarrow s$ can possibly
 be $2\rightarrow5$, $2\rightarrow6$, $2\rightarrow7$, $2\rightarrow10$, $2\rightarrow11$, or $2\rightarrow12$.
 \begin{enumerate}
  \item[Subcase 2.1] If $P$ begins with $2\rightarrow5$, then $P$ must be different from $2\rightarrow5\rightarrow 1$. Moreover, since the path $2\rightarrow5\rightarrow4\rightarrow1$ is transitive,
  then $P$ must go through 3 to 1 (going to 14 is not an option since $P$ would never be able to return to the horizontal line the vertex 1 is on). But then the subdivision of the cell 14(14) has the orientation of type $c$, while 
  $4\rightarrow1$ is not a boundary edge, contradicting with the definition
  of a smart orientation.
  \item[Subcase 2.2] If $P$ begins with $2\rightarrow6$, then the subdivision of the cell 127
  has the orientation of type $a$ or $c$.
  For the orientation of type $a$, $P$ is of length 2, contradiction. For the orientation of type $c$, 
  since the path $2\rightarrow7\rightarrow6\rightarrow1$ is transitive,
  $P$ must go through $8$ to $1$. Thus the subdivision of the cell 17(13)
  has the orientation of type $A$, while $7\rightarrow1$ is not a boundary edge, contradicting with the definition
  of a smart orientation.
  \item[Subcase 2.3] If $P$ begins with $2\rightarrow7$, then $P$ must go through $8$ to $1$. Similarly with the discussion in
  Subcase 2.2, it contradicts with the definition
  of a smart orientation.
  \item[Subcase 2.4] If $P$ begins with $2\rightarrow10$, $P$ must go through the arc $4\rightarrow1$ or $4\rightarrow3$, while either of them
  indicates that $2\rightarrow4$ is not a boundary edge, contradicting the subdivision of the cell 24(15).
  \item[Subcase 2.5] If $P$ begins with  $2\rightarrow11$, it indicates that the orientation of the subdivision of the cell 2(15)(16) is of type $B$.
  Thus the vertex 11 is a sink, contradiction.
  \item[Subcase 2.6] If $P$ begins with  $2\rightarrow12$, the subdivision of the cell 27(17) can possibly has the orientation of type $B$ or
  $C$. In the orientation of type $B$, the vertex 12 is a sink, contradiction. Thus the orientation is of type $C$,
  and $P$ must go through the arc $7\rightarrow1$ or $7\rightarrow8$, while either of them
  indicates that $7\rightarrow1$ is not a boundary edge, contradicting the subdivision of the cell 27(17).
 \end{enumerate}

 Thus, $e$ is not a shortcut in this case.

 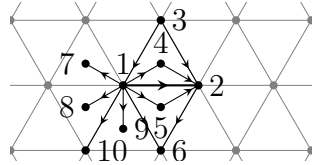
\begin{figure}[!htbp]
 \begin{center}
\begin{tikzpicture}[scale=1]
\clip (-1.5,-1) rectangle (2.5,1.1);
\foreach \x in {-5,...,5} 
 \foreach \y in {-5,...,5}
 { \fill[gray!100] (\x + .5* \y, 0.866*\y) circle(0.3ex);
 }
\foreach \x in {-5,...,5} 
{ \draw[gray,very thin] (\x,-0.866*2)-- +(10,17.32);
  \draw[gray,very thin] (\x,-0.866*2)-- +(-10,17.32);
}
\foreach \y in {-5,...,5} 
{ \draw[gray,very thin] (-3,0.866*\y)-- +(15,0);
}

\draw[, directed] (0,0)-- +(0.5,0.2887);
\draw[,reverse directed] (1,0)-- +(-0.5,0.2887);
\draw[thick,directed] (0,0)-- +(1,0);
\draw[,reverse directed] (0,0)-- +(0.5,0.866);
\draw[,directed] (0,0)-- +(-0.5,0.2887);
\draw[,directed] (0,0)-- +(-0.5,-0.2887);
\draw[,directed] (0,0)-- +(-0.5,-0.866);
\draw[,directed] (0,0)-- +(0,-0.5774);
\draw[,reverse directed] (1,0)-- +(-0.5,0.866);
\draw[,directed] (0,0)-- +(0.5,-0.866);
\draw[,directed] (1,0)-- +(-0.5,-0.866);
\draw[,directed] (0,0)-- +(0.5,-0.2887);
\draw[,reverse directed] (1,0)-- +(-0.5,-0.2887);
\fill[black!100] (0,0) node[above]{1} circle(0.3ex)
                   (0.5,0.2887) node[above]{4} circle(0.3ex)
                   (0.5,0.866)node[right]{3}circle(0.3ex)
                   (1,0)node[right]{2}circle(0.3ex)
                   (0.5,-0.2887) node[below]{5} circle(0.3ex)
                   (0.5,-0.866)node[right]{6}circle(0.3ex)
                   (-0.5,0.2887) node[left]{7} circle(0.3ex)
                   (-0.5,-0.2887) node[left]{8} circle(0.3ex)
                   (-0.5,-0.866)node[right]{10}circle(0.3ex)
                   (0,-0.5774) node[right]{9} circle(0.3ex);

\end{tikzpicture}
\caption{\label{arc3} The case when the arc $e=1\rightarrow 2$ is of type $a_3$.}
\end{center}
\end{figure} 

%

 \item[Case 3.] Suppose that the arc $e=1\rightarrow 2$ is of type $a_3$, as shown in Figure \ref{arc3}. 
 If $e$ is a shortcut, then there exists a directed path $P$ of length at least 3 from $1$ to $2$ 
 (this path does not involve $e$). 
 Suppose that $P$ begins with $e'$. Then $e'=1\rightarrow m$ can potentially
 be $1\rightarrow4$, $1\rightarrow5$, $1\rightarrow7$, $1\rightarrow8$, $1\rightarrow9$ or $1\rightarrow10$. 
 However, $e'$ cannot be $1\rightarrow8$, $1\rightarrow9$ or $1\rightarrow10$, because
 in each of these cases $P$ is forced to go through the vertex
 lying on the horizontal grid line below the vertex 2,
 which is impossible by 
 the properties of smart orientations listed above.
 Also, $e'$ cannot be $1\rightarrow7$, since $1\rightarrow7$ must be an arc in subdiveded cell with 
 orientation of type $B$ and therefore $7$ must be a sink, contradiction.
 If $e'$ is $1\rightarrow4$ or $1\rightarrow5$, then $P$ must be a path of length 2, contradiction.
 Thus, $e$ is not a shortcut in this case.

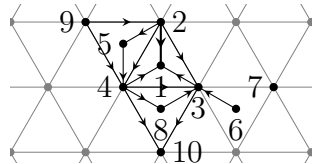
\begin{figure}[!htbp]
 \begin{center}
\begin{tikzpicture}[scale=1]
\clip (-1.5,-1) rectangle (2.5,1.1);
\foreach \x in {-5,...,5} 
 \foreach \y in {-5,...,5}
 { \fill[gray!100] (\x + .5* \y, 0.866*\y) circle(0.3ex);
 }
\foreach \x in {-5,...,5} 
{ \draw[gray,very thin] (\x,-0.866*2)-- +(10,17.32);
  \draw[gray,very thin] (\x,-0.866*2)-- +(-10,17.32);
}
\foreach \y in {-5,...,5} 
{ \draw[gray,very thin] (-3,0.866*\y)-- +(15,0);
}
\draw[,reverse directed] (0,0)-- +(0,0.5774);
\draw[,reverse directed] (0,0)-- +(0.5,0.866);
\draw[,directed] (0,0)-- +(0.5,0.2887);
\draw[,directed] (0,0)-- +(1,0);
\draw[,directed] (0,0)-- +(0.5,-0.2887);
\draw[,reverse directed] (0,0)-- +(-0.5,0.866);
\draw[,directed] (-0.5,0.866)-- +(1,0);
\draw[, directed] (0,0)-- +(0.5,-0.866);

\draw[,directed] (0.5,0.866)-- +(-0.5,-0.2887);
\draw[thick,directed] (0.5,0.866)-- +(0,-0.5774);
\draw[,directed] (0.5,0.866)-- +(0.5,-0.866);

\draw[,directed] (1,0)-- +(-0.5,0.2887);
\draw[,reverse directed] (1,0)-- +(-0.5,-0.2887);
\draw[,reverse directed] (1,0)-- +(0.5,-0.2887);
\draw[,directed] (1,0)-- +(-0.5,-0.866);

 \fill[black!100] (0,0) node[left]{4} circle(0.3ex)
                  (0,0.5774)node[left]{5} circle(0.3ex)
                  (0.5,0.866)node[right]{2}circle(0.3ex) 
                  (-0.5,0.866)node[left]{9}circle(0.3ex)
                  (0.5,-0.866)node[right]{10}circle(0.3ex)  
                  (0.5,0.2887)node[below]{1}circle(0.3ex)
                  (0.5,-0.2887)node[below]{8}circle(0.3ex)
                  (2,0) node[left]{7} circle(0.3ex)
                  (1,0)node[below]{3}circle(0.3ex)
                   +(0.5,-0.2887)node[below]{6}circle(0.3ex);

\end{tikzpicture}
\caption{\label{arc4} The case when the arc $e=2\rightarrow 1$ is of type $a_4$.}
\end{center}
\end{figure}

 \item[Case 4.] Suppose that the arc $e=2\rightarrow 1$ is of type $a_4$, as shown in Figure \ref{arc4}. 
 If $e$ is a shortcut, then there exists a directed path $P$ of length at least 3 from $2$ to $1$ 
 (this path does not involve $e$). Then $P$ can possibly end with 
 $3\rightarrow1$ or $4\rightarrow1$.
 There are two subcases:
 \begin{enumerate}
  \item[Subcase 4.1]  For ending with $3\rightarrow1$, the cell $234$ has the orientation $B$, so $2\rightarrow3$ is a boundary edge. 
  If $4\rightarrow3$ lies on $P$, then $5\rightarrow4$ and $2\rightarrow5$ must lie on $P$ (since the path $2\rightarrow 
  4\rightarrow 1 \rightarrow 3$ is transitive 
  and no path goes from right to left with Euclidean distance larger than 1).
  Hence the subdivided cell $249$
  has the orientation of type $a$ while $2\rightarrow4$ is not a boundary edge, contradicting the definition of a smart orientation.
  The arc $6\rightarrow3$ cannot lie on $P$, since existence of $6\rightarrow3$ implies vertex $6$ is a source by the definition of a smart orientation.
  If $8\rightarrow3$ lies on $P$ then so does $4\rightarrow8$, hence the cell containing $8$ has the orientation $c$, while $3\rightarrow4$ is not a boundary edge, 
  contradicting the definition of a smart orientation.
  \item[Subcase 4.2]  For ending with $4\rightarrow1$, since there is no directed path that goes from right to left with Euclidean distance larger than 1 in a smart 
  orientation, both $5\rightarrow4$ and $2\rightarrow5$ must lie on $P$, and hence the subdivided cell containing vertex $5$
  has the orientation of type $a$, while $2\rightarrow4$ is not a boundary edge, contradicting the definition of a smart orientation.
 \end{enumerate}
 
 Thus, $e$ is not a shortcut in this case. 

 \begin{figure}[!htbp]
 \begin{center}
\begin{tikzpicture}[scale=1]
\clip (-1.5,-1) rectangle (2.5,1.1);
\foreach \x in {-5,...,5} 
 \foreach \y in {-5,...,5}
 { \fill[gray!100] (\x + .5* \y, 0.866*\y) circle(0.3ex);
 }
\foreach \x in {-5,...,5} 
{ \draw[gray,very thin] (\x,-0.866*2)-- +(10,17.32);
  \draw[gray,very thin] (\x,-0.866*2)-- +(-10,17.32);
}
\foreach \y in {-5,...,5} 
{ \draw[gray,very thin] (-3,0.866*\y)-- +(15,0);
}

\draw[thick,reverse directed] (0,0)-- +(0.5,0.2887);
\draw[,reverse directed] (1,0)-- +(-0.5,0.2887);
\draw[,directed] (0,0)-- +(1,0);

\draw[,reverse directed] (0,0)-- +(0.5,0.866);
\draw[,reverse directed] (1,0)-- +(-0.5,0.866);
\draw[,directed] (0,0)-- +(0.5,-0.866);
\draw[,directed] (1,0)-- +(-0.5,-0.866);
\draw[,directed] (1,0)-- +(-0.5,-0.2887);

\fill[black!100] (0,0) node[left]{1} circle(0.3ex)
                   (0.5,0.2887) node[above]{2} circle(0.3ex)
                   (0.5,0.866)node[right]{5}circle(0.3ex)
                   (1,0)node[right]{3}circle(0.3ex)
                   (0.5,-0.2887) node[below]{4} circle(0.3ex)
                   (0.5,-0.866)node[right]{6}circle(0.3ex);

\end{tikzpicture}
\caption{\label{arc5} The case when the arc $e=2\rightarrow 1$ is of type $a_5$.}
\end{center}
\end{figure}
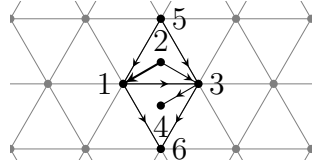

 \item[Case 5.] Suppose that the arc $e=2\rightarrow 1$ is of type $a_5$, as shown in Figure \ref{arc5}. 
 If $e$ is a shortcut, then there exists a directed path $P$ of length at least 3 from $2$ to $1$ 
 (this path does not involve $e$). Then $P$ must begin with the arc $2\rightarrow3$. However, 
 there is no path going from right to left with Euclidean distance larger than 1, thus, $e$ is not a shortcut. 
 
 \begin{figure}[!htbp]
 \begin{center}
\begin{tikzpicture}[scale=1]
\clip (-1.5,-1) rectangle (2.5,1.3);
\foreach \x in {-5,...,5} 
 \foreach \y in {-5,...,5}
 { \fill[gray!100] (\x + .5* \y, 0.866*\y) circle(0.3ex);
 }
\foreach \x in {-5,...,5} 
{ \draw[gray,very thin] (\x,-0.866*2)-- +(10,17.32);
  \draw[gray,very thin] (\x,-0.866*2)-- +(-10,17.32);
}
\foreach \y in {-5,...,5} 
{ \draw[gray,very thin] (-3,0.866*\y)-- +(15,0);
}

\draw[,reverse directed] (0,0)-- +(0.5,0.2887);
\draw[thick,reverse directed] (1,0)-- +(-0.5,0.2887);
\draw[,directed] (0,0)-- +(1,0);

\draw[,reverse directed] (0,0)-- +(0.5,0.866);
\draw[,reverse directed] (1,0)-- +(-0.5,0.866);
\draw[,directed] (0,0)-- +(0.5,-0.866);
\draw[,directed] (1,0)-- +(-0.5,-0.866);
\draw[,directed] (0,0)-- +(0.5,-0.2887);
\draw[,reverse directed] (1,0)-- +(-0.5,-0.2887);

\fill[black!100] (0,0) node[left]{3} circle(0.3ex)
                   (0.5,0.2887) node[above]{2} circle(0.3ex)
                   (0.5,0.866)node[right]{5}circle(0.3ex)
                   (1,0)node[right]{1}circle(0.3ex)
                   (0.5,-0.2887) node[below]{4} circle(0.3ex)
                   (0.5,-0.866)node[right]{6}circle(0.3ex);

\end{tikzpicture}
\caption{\label{arc6} The case when the arc $e=2\rightarrow 1$ is of type $a_6$.}
\end{center}
\end{figure}
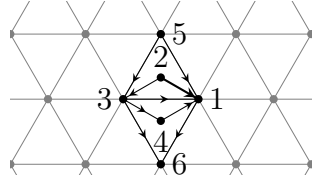

\item[Case 6.] Suppose that the arc $e=2\rightarrow 1$ is of type $a_6$, as shown in Figure \ref{arc6}. 
 If $e$ is a shortcut, then there exists a directed path $P$ of length at least 3 from $2$ to $1$ 
 (this path does not involve $e$). Then $P$ must begin with the arc $2\rightarrow3$. There are only two possibilities here: $P$ is either $2\rightarrow 3\rightarrow 1$, or $2\rightarrow 3\rightarrow 4\rightarrow 1$. In the former case, we do not have a shortcut, while in the later case there is a contradiction with orientation of the cell 136, since the arc $3\rightarrow 1$ is not boundary.
 Thus, $e$ is not a shortcut in this case.

 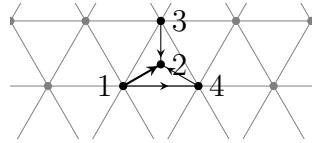
\begin{figure}[!htbp]
 \begin{center}
\begin{tikzpicture}[scale=1]
\clip (-1.5,-0.7) rectangle (2.5,1.1);
\foreach \x in {-5,...,5} 
 \foreach \y in {-5,...,5}
 { \fill[gray!100] (\x + .5* \y, 0.866*\y) circle(0.3ex);
 }
\foreach \x in {-5,...,5} 
{ \draw[gray,very thin] (\x,-0.866*2)-- +(10,17.32);
  \draw[gray,very thin] (\x,-0.866*2)-- +(-10,17.32);
}
\foreach \y in {-5,...,5} 
{ \draw[gray,very thin] (-3,0.866*\y)-- +(15,0);
}
\draw[,reverse directed] (0.5,0.2887)-- +(0,0.5774);
\draw[thick,reverse directed] (0.5,0.2887)-- +(-0.5,-0.2887);
\draw[,reverse directed] (0.5,0.2887)-- +(0.5,-0.2887);
\draw[,directed] (0,0)-- +(1,0);

 \fill[black!100] (0.5,0.2887) node[right]{2} circle(0.3ex)
                 +(0,0.5774)node[right]{3} circle(0.3ex)               
                 +(-0.5,-0.2887)node[left]{1}circle(0.3ex)
                 +(0.5,-0.2887)node[right]{4}circle(0.3ex);

\end{tikzpicture}
\caption{\label{arc7} The case when the arc $e=1\rightarrow 2$ is of type $a_7$.}
\end{center}
\end{figure}

\item[Case 7.] Suppose that the arc $e=1\rightarrow 2$ is of type $a_7$, shown in Figure \ref{arc7}. 
 If $e$ is a shortcut, then there exists a directed path $P$ of length at least 3 from $1$ to $2$ 
 (this path does not involve $e$). Now $P$ 
 can possibly end with $3\rightarrow2$ or $4\rightarrow2$. 
 Since the vertex 3 lies on a horizontal line that is higher than the horizontal line the vertex 1 lies on, the case of $3\rightarrow2$ is impossible.
On the other hand, showing that the case of $4\rightarrow2$ is impossible is similar to our considerations in Case 5.

 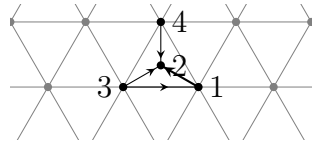
\begin{figure}[!htbp]
 \begin{center}
\begin{tikzpicture}[scale=1]
\clip (-1.5,-0.7) rectangle (2.5,1.1);
\foreach \x in {-5,...,5} 
 \foreach \y in {-5,...,5}
 { \fill[gray!100] (\x + .5* \y, 0.866*\y) circle(0.3ex);
 }
\foreach \x in {-5,...,5} 
{ \draw[gray,very thin] (\x,-0.866*2)-- +(10,17.32);
  \draw[gray,very thin] (\x,-0.866*2)-- +(-10,17.32);
}
\foreach \y in {-5,...,5} 
{ \draw[gray,very thin] (-3,0.866*\y)-- +(15,0);
}
\draw[,reverse directed] (0.5,0.2887)-- +(0,0.5774);
\draw[,reverse directed] (0.5,0.2887)-- +(-0.5,-0.2887);
\draw[thick,reverse directed] (0.5,0.2887)-- +(0.5,-0.2887);
\draw[,directed] (0,0)-- +(1,0);

\fill[black!100] (0.5,0.2887) node[right]{2} circle(0.3ex)
                 +(0,0.5774)node[right]{4} circle(0.3ex)               
                 +(-0.5,-0.2887)node[left]{3}circle(0.3ex)
                 +(0.5,-0.2887)node[right]{1}circle(0.3ex);

\end{tikzpicture}
\caption{\label{arc8} The case when the arc $e=1\rightarrow 2$ is of type $a_8$.}
\end{center}
\end{figure}
 
 \item[Case 8.] Suppose that the arc $e=1\rightarrow 2$ is of type $a_8$, as shown in Figure \ref{arc8}. 
  If $e$ is a shortcut, then there exists a directed path $P$ of length at least 3 from $1$ to $2$ 
 (this path does not involve $e$). Now, $P$ 
 can possibly end with $3\rightarrow2$ or $4\rightarrow2$. 
 Both of these situations are impossible, showing which is similar to our considerations in Case 7.
\end{enumerate}
We are done. \end{proof}

By Lemmas~\ref{LemSmartOri-cycle-free} and~\ref{LemSmartOri}, 
any smart orientation of the boundary subdivision of a triangular grid graph is semi-transitive. Therefore, Theorem~\ref{subdivtrigrid}  is true by Theorem~\ref{semitra}.

\section{Applications of our main result}\label{sec5}

In this section, we consider two applications of Theorem~\ref{subdivtrigrid}. Namely, in Subsection~\ref{sub-tri-eq-tr-shape} we discuss word-representability of subdivisions of triangular grid graphs having equilateral triangle shape, and in Subsection~\ref{sub-Sie-gask-graph} we discuss word-representability of subdivisions of the Sierpi\'{n}ski gasket graph.

\subsection{Subdivision of triangular grid graphs having equilateral triangle shape}\label{sub-tri-eq-tr-shape}

Let $T_n$ be the triangular grid graph shown schematically in Figure~\ref{Tn}. $T_n$ has equilateral triangle shape, and we say that $T_n$ has $n$ levels, that is, $n$ horizontal lines are involved in defining $T_n$. 

\begin{figure}[!htbp]
 \begin{center}
  \begin{tikzpicture}[scale=0.7]
  \fill[black!100] (0,3) circle(0.3ex);
\foreach \y in {1,...,3}
{  \draw[,>=stealth] (0,3)--++(-0.5774*\y,-1*\y);
  \fill[black!100] (-0.5774*\y,3-1*\y) circle(0.3ex);
  }
\foreach \y in {1,...,3}
{
\draw[,>=stealth] (0,3)--++(0.5774*\y,-1*\y);
\fill[black!100] (0.5774*\y,3-1*\y) circle(0.3ex);
};
\foreach \x in {1,...,3} 
 { \draw[,>=stealth] (-0.5774*3,0) -- +(1.1547*\x,0);
    \fill[black!100] (-0.5774*3+1.1547*\x,0) circle(0.3ex);
  }
\draw[,>=stealth] (-0.5774,2) -- +(1.1547,0);
\foreach \x in {1,2} 
 \draw[,>=stealth] (-0.5774*2,1) -- +(1.1547*\x,0);
    \fill[black!100] (0,1) circle(0.3ex);
\foreach \y in {1,...,2}
  \draw[,>=stealth] (0.5774,2)--++(-0.5774*\y,-1*\y);

\foreach \y in {1,...,2}
\draw[,>=stealth] (-0.5774,2)--++(0.5774*\y,-1*\y);

\draw[,>=stealth] (0.5774*2,1)--++(-0.5774,-1);
\draw[,>=stealth] (-0.5774*2,1)--++(0.5774,-1);

\fill[black!100] (-0.5774*4,-1) circle(0.3ex);

\draw[,>=stealth] (-0.5774*4,-1) -- +(1.1547,0);
\draw[,>=stealth] (-0.5774*4,-1) -- +(-0.5774,-1);
\draw[,>=stealth] (-0.5774*4,-1) -- +(0.5774,-1);

\draw[,>=stealth] (-0.5774*2,-1) -- +(-0.5774,-1);
\draw[,>=stealth] (-0.5774*2,-1) -- +(0.5774,-1);

\draw[,>=stealth] (0.5774*4,-1) -- +(-0.5774,-1);
\draw[,>=stealth] (0.5774*4,-1) -- +(0.5774,-1);
\draw[,>=stealth] (0.5774*2,-1) -- +(1.1547,0);
\draw[,>=stealth] (0.5774*2,-1) -- +(-0.5774,-1);
\draw[,>=stealth] (0.5774*2,-1) -- +(0.5774,-1);

\foreach \x in {-5,-3,1,3}
 \draw[,>=stealth] (0.5774*\x,-2) -- +(0.5774*2,0);
\foreach \x in {-4,-2,...,4}
{\fill[black!100] (-0.5774*\x,-1) circle(0.3ex);
\fill[black!100] (-0.5774*\x-0.5774,-2) circle(0.3ex);}
  \fill[white!100] (0,-1) circle(0.4ex);
\fill[black!100] (0.5774*5,-2) circle(0.3ex);
\draw[style=loosely dotted, thick] (0.5774,0)--(1.1547,-1) [xshift=1.1547cm] (0.5774,0)--(1.1547,-1);
\draw[style=loosely dotted, thick] (-0.5774,0)--(-1.1547,-1) [xshift=-1.1547cm] (-0.5774,0)--(-1.1547,-1);

\draw[style=loosely dotted, thick] (-1.1547*0.5,-0.5)--(1.1547*0.5,-0.5);
\draw[style=loosely dotted, thick] (-1.1547,-1)--(1.1547,-1);
\draw[style=loosely dotted, thick] (-1.1547*.5,-2)--(1.1547*.5,-2);
\end{tikzpicture}
\caption{\label{Tn} The triangular grid graph $T_n$ having equilateral triangle shape with $n$ levels.}
 \end{center}
\end{figure}
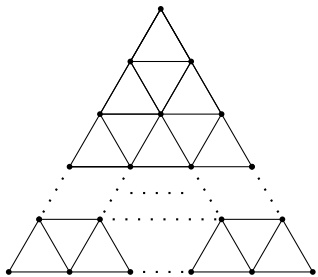

It follows from Theorem~\ref{subdivtrigrid} that subdividing an interior cell in $T_n$ will result in 
a non-word-representable graph. Let $A_n$ be the graph obtained from $T_n$ by subdividing all of its 
boundary cells, that is, $A_n$ is the maximum subdivision of $T_n$. Again, by Theorem~\ref{subdivtrigrid},
$A_n$ is word-representable, and an example of a smart (semi-transitive) orientation is presented in 
Figure~\ref{AnOr}, where we also indicate types of orientations of cells used.

\newsavebox{\AnOri}
\newsavebox{\TriType}
\savebox{\AnOri}
{
\begin{tikzpicture}[scale=0.7]
  \fill[black!100] (0,3) circle(0.3ex);
\foreach \y in {1,...,3}
{  
  \fill[black!100] (-0.5774*\y,3-1*\y) circle(0.3ex);
}
\foreach \y in {1,...,3}
{
\fill[black!100] (0.5774*\y,3-1*\y) circle(0.3ex);
};
\foreach \x in {1,...,3} 
 { 
    \fill[black!100] (-0.5774*3+1.1547*\x,0) circle(0.3ex);
  }
\draw[,directed] (-0.5774,2) -- +(1.1547,0);
    \fill[black!100] (0,1) circle(0.3ex);
%

\foreach \y in {0, -1,-2}
{
\draw[,directed,shift={(0.5774*\y,\y)}] (0,3)--++(-0.5774,-1);
\draw[,directed,shift={(-0.5774*\y,\y)}] (0,3)--++(0.5774,-1);
\draw[,directed,shift={(1.1547*\y,0)}] (0.5,0)--++(1.1547,0);
}
\foreach \y in {0, 1}
{
\draw[,directed,shift={(0.5774*\y,\y)}] (0,1)--++(-0.5774,-1);
\draw[,directed,shift={(-0.5774*\y,\y)}] (0,1)--++(0.5774,-1);
}
\draw[,directed] (-1.1547,1)--++(1.1547,0);
\draw[,directed,shift={(1.1547,0)}] (-1.1547,1)--++(1.1547,0);
\draw[,directed] (0.5774*2,1)--++(-0.5774,-1);
\draw[,directed] (-0.5774*2,1)--++(0.5774,-1);

\fill[black!100] (-0.5774*4,-1) circle(0.3ex);

\draw[,directed] (-0.5774*4,-1) -- +(1.1547,0);
\draw[,directed] (-0.5774*4,-1) -- +(-0.5774,-1);
\draw[,directed] (-0.5774*4,-1) -- +(0.5774,-1);

\draw[,directed] (-0.5774*2,-1) -- +(-0.5774,-1);
\draw[,directed] (-0.5774*2,-1) -- +(0.5774,-1);

\draw[,directed] (0.5774*4,-1) -- +(-0.5774,-1);
\draw[,directed] (0.5774*4,-1) -- +(0.5774,-1);
\draw[,directed] (0.5774*2,-1) -- +(1.1547,0);
\draw[,directed] (0.5774*2,-1) -- +(-0.5774,-1);
\draw[,directed] (0.5774*2,-1) -- +(0.5774,-1);

\foreach \x in {-5,-3,1,3}
 \draw[,directed] (0.5774*\x,-2) -- +(0.5774*2,0);
\foreach \x in {-4,-2,...,4}
{\fill[black!100] (-0.5774*\x,-1) circle(0.3ex);
\fill[black!100] (-0.5774*\x-0.5774,-2) circle(0.3ex);}
  \fill[white!100] (0,-1) circle(0.4ex);

\draw[style=loosely dotted, thick] (0.5774,0)--(1.1547,-1) [xshift=1.1547cm] (0.5774,0)--(1.1547,-1);
\draw[style=loosely dotted, thick] (-0.5774,0)--(-1.1547,-1) [xshift=-1.1547cm] (-0.5774,0)--(-1.1547,-1);

\draw[style=loosely dotted, thick] (-1.1547*0.5,-0.5)--(1.1547*0.5,-0.5);
\draw[style=loosely dotted, thick] (-1.1547,-1)--(1.1547,-1);
\draw[style=loosely dotted, thick] (-1.1547*.5,-2)--(1.1547*.5,-2);

 \draw[,directed] (0,3)--+(0,-2/3);
  \draw[,directed] (0,3)++(0,-2/3)-- +(0.5774,-1/3);
   \draw[,directed] (0,3)++(0,-2/3)-- +(-0.5774,-1/3);

\fill[black!100] (0,2+1/3) circle(0.3ex);
\fill[black!100] (0.5774*5,-2) circle(0.3ex);
\foreach \y in {1,2,4} 
{\fill[black!100] [shift={(0.5774*\y,-\y)}] (0,2+1/3) circle(0.3ex);
 \fill[black!100] [shift={(-0.5774*\y,-\y)}] (0,2+1/3) circle(0.3ex);
  \draw[,reverse directed,shift={(0.5774*\y,-\y)}] (0,2+1/3)-- +(0,2/3);  
  \draw[,reverse directed,shift={(-0.5774*\y,-\y)}] (0,2+1/3)-- +(0,2/3);  
  \draw[,directed,shift={(-0.5774*\y,-\y)}] (0,2+1/3)-- +(-0.5774,-1/3); 
  \draw[,directed,shift={(-0.5774*\y,-\y)}] (0,2+1/3)-- +(0.5774,-1/3);  
  \draw[,reverse directed,shift={(0.5774*\y,-\y)}] (0,2+1/3)-- +(-0.5774,-1/3); 
  \draw[,reverse directed,shift={(0.5774*\y,-\y)}] (0,2+1/3)-- +(0.5774,-1/3);  
 }
\foreach \x in {-1,1}
{ \fill[black!100] [shift={(1.1547*\x,0)}] (0,-2+1/3) circle(0.3ex);
  \draw[,reverse directed,shift={(1.1547*\x,0)}] (0,-2+1/3)-- +(0,2/3);  
  \draw[,reverse directed,shift={(1.1547*\x,0)}] (0,-2+1/3)-- +(-0.5774,-1/3); 
  \draw[,directed,shift={(1.1547*\x,0)}] (0,-2+1/3)-- +(0.5774,-1/3);  
}
\end{tikzpicture}
}
\savebox{\TriType}
{
\begin{tikzpicture}[scale=0.7]
\foreach \y in {0,-1.8,-3.6}{
 \draw[,directed,yshift=\y cm] (0,1)-- +(-0.5774,-1) ;
 \draw[,directed,yshift=\y cm] (0,1)-- +(0.5774,-1);
 \draw[,directed,yshift=\y cm] (0,1)-- +(0,-2/3); 
 \draw[,directed,yshift=\y cm] (-0.5774,0)--(0.5774,0);
 \fill[black!100,yshift=\y cm] (0,1) circle(0.4ex)
                 +(-0.5774,-1) circle(0.4ex)
                 +(0.5774,-1) circle(0.4ex)
                 +(0,-2/3) circle(0.4ex);
 
}
 \draw[,reverse directed] (-0.5774,0)-- +(0.5774,1/3);
  \draw[,directed,yshift=-1.8 cm] (-0.5774,0)-- +(0.5774,1/3);
   \draw[,directed,yshift=-1.8*2 cm] (-0.5774,0)-- +(0.5774,1/3);
 \draw[,reverse directed] (0.5774,0)-- +(-0.5774,1/3);
  \draw[,directed,yshift=-1.8 cm] (0.5774,0)-- +(-0.5774,1/3);
   \draw[,reverse directed,yshift=-1.8*2 cm] (0.5774,0)-- +(-0.5774,1/3);
 \path (0,-.3) node {\footnotesize$A$}
       ++(0,-1.8) node {\footnotesize$B$}
       +(0,-1.8) node  {\footnotesize$C$};
\end{tikzpicture}
}
 
 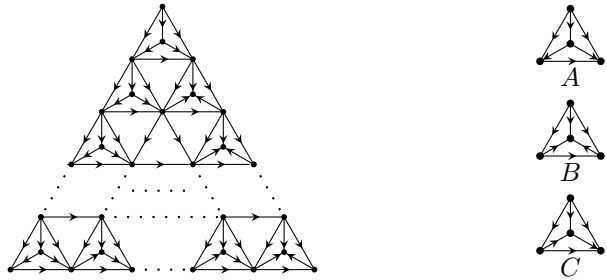
\begin{figure}[!htbp]
 \begin{center}
\begin{picture}(250,100)
\put(0,0){\usebox{\AnOri}}
\put(200,-5){\usebox{\TriType}}
\end{picture}
\caption{\label{AnOr} Semi-transitive orientation of $A_n$.}
\end{center}
\end{figure}

\subsection{Subdivisions of the Sierpi\'{n}ski gasket graph}\label{sub-Sie-gask-graph}

For the two-dimensional Sierpi\'{n}ski gasket graph $SG(n)$, by Theorem~\ref{subdivtrigrid}, we can obtain 
its maximum word-representable subdivision by subdividing all of its
boundary cells. 
Figure \ref{SubdivSG2o} shows the maximum word-representable subdivision of $SG(3)$ and one of its smart orientations.

\newsavebox{\AS}
\savebox{\AS}
{\begin{tikzpicture}[scale=0.5]

\draw (0,0) -- ++(0.5774,1) --++(0.5774,1)
                   --++(0.5774,-1) --++(0.5774,-1)
                   --++(-2*0.5774,0)--++(-0.5774,1)--++(2*0.5774,0)
                   --++(-0.5774,-1)--+(-2*0.5774,0)
                   (2*0.5774,0);
\draw (0,0) -- ++(0.5774,1/3) --++(0,2/3)-- ++(0.5774,1/3)
                    --++(0,2/3)++(0,-2/3)
                   -- ++(0.5774,-1/3) --++(0,-2/3)-- ++(0.5774,-1/3)++(-0.5774,1/3)
                   -- ++(-0.5774,-1/3) -- ++(-0.5774,1/3);
          
\fill[black!100] (0,0) circle(0.5ex)
                 ++(0.5774,1) circle(0.5ex)
                 ++(0.5774,1) circle(0.5ex)
                 ++(0.5774,-1) circle(0.5ex)
                 ++(0.5774,-1) circle(0.5ex)
                 ++(-2*0.5774,0)circle(0.5ex)
               ;
\fill[black!100] (0,0)
                 ++(0.5774,1/3) circle(0.5ex)
                 ++(0.5774,1) circle(0.5ex)
                 ++(0.5774,-1) circle(0.5ex)
               ;
\end{tikzpicture}
}
%
%
%

$SG(n)$ can only have faces of degree $3\cdot 2^k$, where $k=0,1,\ldots$, and the operation 
of subdivision of a (triangular) cell can be generalized to subdivision of other faces. 
One such generalization is inserting a new node inside a face and connecting it to 
{\em all} of the face's nodes. Another possible generalization is subdividing a face 
into three parts by connecting a newly added node to the three {\em conner nodes} of a 
face (note that each face being a $3\cdot 2^k$-cycle, looks like a triangle, and the 
conner nodes are the vertices of such a triangle). 

\newsavebox{\ASO}
\savebox{\ASO}
{
\begin{tikzpicture}[scale=0.5]
\foreach \x in {0,1}
{
 \draw[,directed] (0.5774*\x,\x)-- +(2*0.5774,0);
 \draw[,reverse directed] (0.5774*\x,\x)-- +(0.5774,1/3);
 \draw[,reverse directed] (0.5774*\x,\x)-- +(0.5774,1);
 \draw[,directed,shift={(0.5774*\x,\x)}] (0.5774,1)-- +(0.5774,-1);
  \draw[,directed,shift={(0.5774*\x,\x)}] (0.5774,1)-- +(0,-2/3);
 \draw[,directed,shift={(0.5774*\x,\x)}] (0.5774,1/3)-- +(0.5774,-1/3);  
}

 \draw[,directed] (0.5774*2,0)-- +(2*0.5774,0);
 \draw[,directed] (0.5774*2,0)-- +(0.5774,1/3);
 \draw[,reverse directed] (0.5774*2,0)-- +(0.5774,1);
 \draw[,directed,shift={(0.5774*2,0)}] (0.5774,1)-- +(0.5774,-1);
  \draw[,directed,shift={(0.5774*2,0)}] (0.5774,1)-- +(0,-2/3);
 \draw[,reverse directed,shift={(0.5774*2,0)}] (0.5774,1/3)-- +(0.5774,-1/3);  

\fill[black!100] (0,0) circle(0.5ex)
                 ++(0.5774,1) circle(0.5ex)
                 ++(0.5774,1) circle(0.5ex)
                 ++(0.5774,-1) circle(0.5ex)
                 ++(0.5774,-1) circle(0.5ex)
                 ++(-2*0.5774,0)circle(0.5ex)
                 ;
\fill[black!100] (0,0)
                 ++(0.5774,1/3) circle(0.5ex)
                 ++(0.5774,1) circle(0.5ex)
                 ++(0.5774,-1) circle(0.5ex)
               ;
\end{tikzpicture}

}

%

\begin{figure*}[!htbp]
\begin{center}
\begin{tikzpicture}[scale=0.5]
\path (0,0) node {\usebox{\AS}}
      +(0.5774*2,2) node {\usebox{\AS}}
      +(0.5774*4,0) node {\usebox{\AS}}
      ;
\path (0.5774*4,4) node {\usebox{\AS}}
      +(0.5774*2,2) node {\usebox{\AS}}
      +(0.5774*4,0) node {\usebox{\AS}}
      ; 
\path (0.5774*8,0) node {\usebox{\AS}}
      +(0.5774*2,2) node {\usebox{\AS}}
      +(0.5774*4,0) node {\usebox{\AS}}
      ;      

\path (11,0)++(0,0) node {\usebox{\ASO}}
      +(0.5774*2,2) node {\usebox{\ASO}}
      +(0.5774*4,0) node {\usebox{\ASO}}
      ;
\path (11,0)++(0.5774*4,4) node {\usebox{\ASO}}
      +(0.5774*2,2) node {\usebox{\ASO}}
      +(0.5774*4,0) node {\usebox{\ASO}}
      ; 
\path (11,0)++(0.5774*8,0) node {\usebox{\ASO}}
      +(0.5774*2,2) node {\usebox{\ASO}}
      +(0.5774*4,0) node {\usebox{\ASO}}
      ; 
\end{tikzpicture}
\caption{\label{SubdivSG2o}The maximum word-representable subdivision of $SG(3)$ and one of its smart orientations.}
\end{center}
\end{figure*}
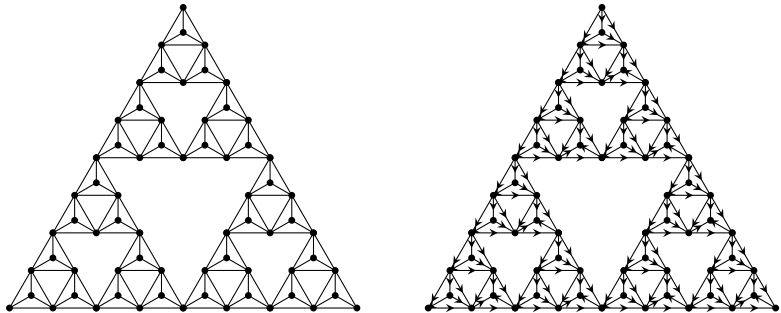

We leave it as an open problem to study word-representability of  the Sierpi\'{n}ski gasket graph 
when subdivision of  faces of larger degrees is allowed.

\end{document}